\documentclass[10pt,a4paper]{article}
    \usepackage{amsmath,,amssymb,amsthm}
    \usepackage{mathtools}
    \usepackage{enumerate,enumitem}
    \usepackage[all]{xy}
    \usepackage[margin=25mm]{geometry}
    \usepackage{titlesec}
        \titleformat{\section}{\normalfont\large\bf}{\thesection.}{1ex}{\centering}
        \titleformat{\subsection}[runin]{\normalfont\bf}{\thesubsection.}{1ex}{}[.]
    \theoremstyle{plain}
        \newtheorem{theorem}{Theorem}[section]
        \newtheorem{lemma}[theorem]{Lemma}
        
        \newtheorem{corollary}[theorem]{Corollary}
    \theoremstyle{definition}
        
    \theoremstyle{remark}
        
    \numberwithin{equation}{section}
    \DeclareMathOperator*{\esssup}{ess\,sup}
    \DeclareMathOperator{\supp}{supp}
\newcommand{\e}{\epsilon}
\newcommand{\R}{\mathbb{R}}
\begin{document}
\begin{center}\large
    {\bf Properties of Navier-Stokes mild solutions with \\ initial data in subcritical Lorentz spaces} \\\ \\
    Joseph P.\ Davies\footnote{University of Sussex, Brighton, UK, {\em jd535@sussex.ac.uk}} and Gabriel S. Koch\footnote{University of Sussex, Brighton, UK, {\em g.koch@sussex.ac.uk}}
\end{center}
\begin{abstract}
We are interested in mild solutions to the Navier-Stokes equations with initial data $f$ in a subcritical ($p>n$) Lorentz space  $L^{p,q}(\mathbb{R}^{n})$, which, if $p <\infty$, is continuously embedded into the larger  subcritical homogeneous Besov space 
$\dot B^{-\frac np}_{\infty,\infty}(\R^n)$.  For $p\in (n,\infty)$ and $f\in \dot B^{-\frac np}_{\infty,\infty}(\R^n)$, it is known that there exists some maximal time $T^B_f \in (0,\infty]$ such that a unique local (in time) mild solution $u_f^B$ exists in, e.g., $L^\infty(0,T;\dot B^{-\frac np}_{\infty,\infty}(\R^n)) \cap L^{4/(1+\frac np)}(0,T;L^\infty(\R^n))$ for all $T\in (0,T^B_f)$, and if $T^B_f <\infty$, then 
one has the blow-up estimate $\|u^B_f(t)\|_{\dot B^{-\frac np}_{\infty,\infty}(\R^n)} \gtrsim_{n,p} (T^B_f -t)^{-\frac 12 (1-\frac np)}$ for all $t\in (0,T^B_f)$.  The continuous embedding means that one can replace $\dot B^{-\frac np}_{\infty,\infty}(\R^n)$ in the blow-up estimate by $L^{p,q}(\mathbb{R}^{n})$ for any $q\in [1,\infty]$, which would not be very interesting if $\|u(t)\|_{L^{p,q}(\mathbb{R}^{n})} \equiv \infty$; even if this is not the case, in principle  one might be able to further replace $T^B_f$ in the estimate by some smaller $T^L_f \in (0,T^B_f)$.   

For $f\in L^{p,q}(\mathbb{R}^{n})$, we show that this cannot happen:  there exists  a maximal time $T^L_f\in (0,\infty]$ such that there exists a unique local mild solution 
${u_f^L \in L^\infty(0,T;L^{p,q}(\mathbb{R}^{n})) \cap L^{\alpha}(0,T;L^\infty(\R^n))}$ \linebreak for all ${T\in (0,T^L_f)}$ and $\alpha \in [1,\frac{2p}n)$, and if $T^L_f <\infty$, one has the blow-up estimate (established by J. Leray (1934) in the Lebesgue setting $r=s$)   ${\|u_f^L(t)\|_{L^{r,s}(\mathbb{R}^{n})} \gtrsim_{n,r} (T^L_f -t)^{-\frac 12 (1-\frac nr)}}$  for all $t\in (0,T^L_f)$, $r\in (n,\infty]$ and $s\in [1,\infty]$; in particular, $u_f^L \notin L^{4/(1+\frac np)}(0,T^L_f;L^\infty(\R^n))$.   The solution  is shown to be unique in, e.g.,  the larger class $\cup_{\alpha >2}L^\alpha(0,T;L^\infty(\R^n))$, so that $T^L_f$ is not only independent of $p$ and $q$, but also  $ T^L_f=T^B_f $ and $  u_f^L\equiv u_f^B$.  This may be thought of as a `propagation of regularity' type of result for $u_f^B$, in that $\|u_f^B(t)\|_{L^{p,q}(\mathbb{R}^{n})}$ cannot blow up prior to $T^B_f$ for $f\in L^{p,q}(\mathbb{R}^{n})$.  Our existence results are based on the method of T. Kato (1984) with Lorentz spaces replacing Lebesgue spaces throughout, and are given without any reference to the Besov framework.  The uniqueness results are similarly self-contained, extending certain Lebesgue space methods of Fabes-Jones-Riviere (1972) to the Lorentz setting.  We also establish certain continuity properties of the solutions which are constructed.
\end{abstract}


\section{Introduction and context}
For initial data $f$ in a subcritical ($\epsilon >0$) homogeneous Besov space $\dot B^{-1+\e}_{\infty,\infty}(\R^n)$ with $\e\in (0,1)$, there are various ways (see, e.g., \cite{chemin2019a,davies2022,lemarie2016}) to construct a local-in-time `mild' Navier-Stokes solution $u_f^B$ with initial data $u_f^B(0)=f$.  One can moreover show that there exists some  maximal time $T^B_f \in (0,\infty]$ such that $u_f^B$ is the unique such solution on $(0,T^B_f)$, and if $T^B_f<\infty$ then one has the (scaling-invariant) `blow-up estimate'
$$\|u_f^B(t)\|_{\dot B^{-1+\e}_{\infty,\infty}(\R^n)} \gtrsim (T^B_f -t)^{-\frac \e 2} \quad \forall \ t\in (0,T^B_f)\, .$$
Of course, such a blow-up-rate type of result is not very interesting if, e.g.,  $u_f^B(t)\notin \dot B^{-1+\e}_{\infty,\infty}(\R^n)$ for some, or all, $t\in (0,T^B_f)$ (in which case, by convention, we write $\|u_f^B(t)\|_{\dot B^{-1+\e}_{\infty,\infty}(\R^n)} = \infty$).  As we showed in \cite{davies2022}, however, one can  construct $u_f^B$ (and $T^B_f$) in such a way that $u_f^B(t)\in \dot B^{-1+\e}_{\infty,\infty}(\R^n)$ for all $t\in (0,T^B_f)$ and
$$u_f^B\in L^\infty(0,T;\dot B^{-1+\e}_{\infty,\infty}(\R^n))\cap L^{4/(2-\epsilon)}(0,T;L^\infty(\R^n)) \quad \textrm{whenever}\ 0<T<T^B_f$$
and the above blow-up estimate holds, meaning in particular that $T^B_f$, if finite, is the first time the $\dot B^{-1+\e}_{\infty,\infty}(\R^n)$-norm of $u_f^B$ becomes unbounded, and can be thought of as the first loss of the `regularity' it maintains at all previous times through its membership in the space $\dot B^{-1+\e}_{\infty,\infty}(\R^n)$.

One might then wonder about the potential loss of other types of regularity if  the initial data $f$ is more regular.  If, for example, $f$ belongs to the smaller Lorentz space $L^{\frac n{1-\e},q}(\R^n)$ for some $\e \in (0,1)$ and $q\in [1,\infty]$, then by the embedding
$$L^{\frac n{1-\e},q}(\R^n) \hookrightarrow \dot B^{-1+\e}_{\infty,\infty}(\R^n) $$
(see, e.g., \cite{davies2022}) we see that $f\in \dot B^{-1+\e}_{\infty,\infty}(\R^n)$ so that we may take $u_f^B$ and $T^B_f$ as above; moreover, as the embedding is continuous, the blow-up estimate above implies the (again, scaling-invariant) blow-up estimate
$$\|u_f^B(t)\|_{L^{\frac n{1-\e},q}(\R^n)} \gtrsim (T^B_f -t)^{-\frac \e 2} \quad \forall \ t\in (0,T^B_f)\, .$$
As above, however, this would not be very interesting if $u_f^B(t)\notin L^{\frac n{1-\e},q}(\R^n)$ (i.e., ${\|u_f^B(t)\|_{L^{\frac n{1-\e},q}(\R^n)} = \infty}$) for some, or all, $t\in (0,T^B_f)$.  In principle, there could as well exist some $T^L_f \in (0,T^B_f)$ such that 
$$\|u_f^B(t)\|_{L^{\frac n{1-\e},q}(\R^n)} \gtrsim (T^L_f -t)^{-\frac \e 2} \quad \forall \ t\in (0,T^L_f)$$
 while $u_f^B(t)\in L^{\frac n{1-\e},q}(\R^n)$ for all $t\in (0,T^L_f)$ and
$$u_f^B\in L^\infty(0,T;L^{\frac n{1-\e},q}(\R^n)) \quad \textrm{whenever}\ 0<T<T^L_f\, .$$
This means that the $L^{\frac n{1-\e},q}(\R^n)$-norm of $u_f^B$ could, in principle, become unbounded before the \linebreak $\dot B^{-1+\e}_{\infty,\infty}(\R^n)$-norm does, which could be thought of roughly as a `loss of regularity'.\footnote{In principle, nothing so far prevents  $u_f^B(t) \in \dot B^{-1+\e}_{\infty,\infty}(\R^n) \setminus L^{\frac n{1-\e},q}(\R^n)$ for all $t\in (T_f^L,T_f^B)$ as well, so that $u_f^B$ would be `less regular' after $T_f^L$ than before it.}

Here we provide at least one proof that that cannot happen, by showing that  $T^L_f:=T^B_f$ satisfies all of the above properties.  This can be thought of as a `propagation of regularity' type of result in the Besov framework of, e.g., \cite{davies2022}.  Our  method of proof, however, is  simply to provide a suitable self-contained local well-posedness theory in subcritical Lorentz spaces, without any reference to the larger Besov spaces.  In particular, we generalize the existence method of Kato \cite{kato1984} from the Lebesgue setting to the Lorentz setting (including\footnote{For initial data in certain Besov spaces with negative regularity, it is a standard method (see, e.g., \cite{bahouri2011,lemarie2016}) to essentially use Kato's method directly to construct solutions which will live in Lebesgue spaces for positive times.  Here, for subcritical Lorentz intial data, we replace those Lebesgue spaces with more general  Lorentz spaces.} the `path spaces'), while generalizing some methods of Fabes-Jones-Riviere \cite{fabes1972} from the Lebesgue to the Lorentz settings to obtain certain uniqueness theorems which we require.  The specific self-contained local well-posedness theory in Lorentz spaces (including certain continuity properties of solutions) which we provide here may be of independent interest. 

Specifically, for $\e \in (0,1)$, $q\in [1,\infty]$ and initial data $f\in L^{\frac n{1-\e},q}(\R^n)$, we show\footnote{Our full and most general results are given in the next section; here we mention only what is relevant in the context of the Besov theory.}  that  there exists  a maximal time $T^L_f\in (0,\infty]$ such that there exists a  local mild solution $u^L_f$ on $(0,T^L_f)$ with ${u_f^L(t)\in L^{\frac n{1-\e},q}(\R^n)}$ for all $t\in (0,T^L_f)$ which 
satisfies\footnote{In fact, $\sup_{0<t<T}\|u_f^L(t)\|_{L^{\frac n{1-\e},q}(\R^n)} <\infty$ for all $T\in (0,T^L_f)$.}
$${u_f^L \in L^\infty(0,T;L^{\frac n{1-\e},q}(\mathbb{R}^{n})) \cap L^{\alpha}(0,T;L^\infty(\R^n))} \quad \textrm{whenever}\ 0<T<T^L_f$$ 
for all $\alpha \in [1,\frac2{1-\e})$  and  is unique (as is $u_f^B$, see \cite{davies2022}) in the larger class $\cup_{\alpha >2}L^\alpha(0,T;L^\infty(\R^n))$, and if $T^L_f <\infty$, one has the blow-up estimate (established  by  Leray \cite{leray1934} for $n=3$ in the Lebesgue setting $r=s$; see \eqref{lebesgueburate} below)  $${\|u_f^L(t)\|_{L^{r,s}(\mathbb{R}^{n})} \gtrsim_{n,r} (T^L_f -t)^{-\frac 12 (1-\frac nr)}} \quad \forall \ t\in (0,T^L_f)$$ 
for all $r\in (n,\infty]$ and $s\in [1,\infty]$.  In the context of the Besov solution $u_f^B$ of \cite{davies2022} mentioned above, it is crucial that   this more general blow-up estimate  implies (taking ${r=s=\infty}$) that  $u_f^L\notin L^{4/(2-\epsilon)}(0,T^L_f;L^\infty(\R^n))$.   The uniqueness in $\cup_{\alpha >2}L^\alpha(0,T;L^\infty(\R^n))$ then allows one to deduce easily that $T^L_f$ is not only independent of $\epsilon$ and $q$, but also that  ${T^L_f=T^B_f}$ and ${u_f^L\equiv u_f^B}$, thus establishing the above-mentioned propagation of regularity result for $u_f^B$ when ${f\in L^{\frac n{1-\e},q}(\R^n)}$.

\section{Preliminaries and main results}
According to the physical theory, if an incompressible viscous Newtonian fluid occupies the whole space $\mathbb{R}^{n}$ in the absence of external forces, then the velocity $U(\tau,x)$ and kinematic pressure $P(\tau,x)$ of the fluid at time $\tau>0$ and position $x\in\mathbb{R}^{n}$ satisfy the \textbf{\textit{Navier-Stokes equations}}
\begin{equation}\label{navier-stokes-equations}
    \left\{\begin{array}{ll} \partial_{\tau}U_{j}-\nu\Delta U_{j}+\nabla_{k}(U_{j}U_{k}) + \nabla_{j}P = 0 \quad (1\leq j \leq n), \\ \nabla\cdot U=0, \end{array}\right.
\end{equation}
where the summation convention is used over repeated indices, and the coefficient $\nu>0$ is the \textbf{\textit{kinematic viscosity}} of the fluid. If the fluid is known to have a velocity $F(x)$ at the initial time $\tau=0$ (satisfying $\nabla\cdot F=0$), then the pair $(U,P)$ is a solution to the \textbf{\textit{initial value problem}} consisting of \eqref{navier-stokes-equations} with the initial condition $U(0,x)=F(x)$. By considering the rescaled quantities
\begin{equation}
    t = \nu\tau, \quad u(t,x)=\nu^{-1}U(\nu^{-1}t,x), \quad p(t,x)=\nu^{-2}P(\nu^{-1}t,x), \quad f(x) = \nu^{-1}F(x),
\end{equation}
whose physical dimensions are powers of length alone, we may rewrite the initial value problem in the standardised form
\begin{equation}\label{initial-value-problem}
    \left\{\begin{array}{ll} \partial_{t}u_{j}-\Delta u_{j}+\nabla_{k}(u_{j}u_{k}) + \nabla_{j}p = 0 \quad (1\leq j \leq n), \\ \nabla\cdot u=0, \\ u(0)=f.\end{array}\right.
\end{equation}
At the formal level (a rigorous argument in the context of Lorentz spaces, not needed for the purposes of this paper, is given in \cite{davies2021}), we can eliminate the pressure term by applying the Leray projection
\begin{equation}
    \mathbb{P}_{ij} = \delta_{ij} - \frac{\nabla_{i}\nabla_{j}}{\Delta}
\end{equation}
onto solenoidal vector fields, which transforms \eqref{initial-value-problem} into the nonlinear heat equation
\begin{equation}
    \left\{\begin{array}{ll} \partial_{t}u_{i}-\Delta u_{i}+\mathbb{P}_{ij}\nabla_{k}(u_{j}u_{k}) = 0 \quad (1\leq i \leq n), \\ \nabla\cdot u=0, \\ u(0)=f,\end{array}\right.
\end{equation}
for which we have the corresponding integral equation
\begin{equation}\label{integral-equation}
    u_{i}(t,x) = \int_{\mathbb{R}^{n}}\Phi(t,x-y)f_{i}(y)\,\mathrm{d}y - \int_{0}^{t}\int_{\mathbb{R}^{n}}\nabla_{k}\mathcal{T}_{ij}(t-s,x-y)u_{j}(s,y)u_{k}(s,y)\,\mathrm{d}y\,\mathrm{d}s
\end{equation}
($1\leq i \leq n$), where $\Phi$ is the heat kernel and $\mathcal{T}_{ij}=\mathbb{P}_{ij}\Phi$ is the Oseen kernel (see (\ref{heatkerdef}) - (\ref{oseenkerdef}) below). The theory of so-called \textbf{\textit{mild solutions}} is based on the study of the integral equation \eqref{integral-equation}, dating back to the groundbreaking work of Leray \cite{leray1934}.

For $n=3$, Leray \cite{leray1934} proves that if $T\in(0,\infty)$, $f\in L^{\infty}(\mathbb{R}^{3})$ and $T^{\frac{1}{2}}{\|f\|}_{L^{\infty}(\mathbb{R}^3)}<M$ (where $M$ is some dimensionless constant independent of $f$ and $T$), then there exists a bounded continuous function \linebreak $u\in BC\left((0,T)\times\mathbb{R}^3\right)$ satisfying \eqref{integral-equation} for all $(t,x)\in(0,T)\times\mathbb{R}^{3}$. For each $f\in L^{\infty}(\mathbb{R}^3)$, we therefore obtain a solution $u\in BC\left((0,T)\times\mathbb{R}^3\right)$ for small $T$ by making the dimensionless quantity $T^{\frac{1}{2}}{\|f\|}_{L^{\infty}(\mathbb{R}^3)}$ sufficiently small. In fact, for $f\in C^1(\mathbb{R}^3) \cap H^1(\mathbb{R}^3) \cap L^{\infty}(\mathbb{R}^{3})$, Leray shows\footnote{In fact, the assumption on $f$ can be relaxed to $f\in L^{\infty}(\mathbb{R}^{3}) \cap L^{2}(\mathbb{R}^{3})$, see e.g. \cite{ozanski2017}.} that this solution is completely smooth and hence also satisfies (\ref{initial-value-problem}) on $(0,T)\times\mathbb{R}^{3}$ for an associated pressure $p=p(u)$; in particular, \linebreak $u\in C((0,T);L^\infty(\mathbb{R}^3)) \cap C([0,T);L^2(\mathbb{R}^3))$ which implies uniqueness (cf. Theorem \ref{uniqthm} below) for such a fixed $f$ in his class of `regular' solutions.  Leray observed that his existence criterion, along with the above-mentioned uniqueness which implies a semi-group property (cf. Lemma \ref{semigrouplem} below),  implies a certain \textbf{\textit{blowup estimate}} which we can state roughly\footnote{The accurate statement would be obtained by explicitly defining `regular' as Leray does.} as follows: if a regular solution \linebreak $u\in\cap_{T'\in(0,T)} C((0,T');L^\infty(\mathbb{R}^3)) \cap C([0,T'); L^2(\mathbb{R}^3))$ on a finite time interval $(0,T)$ cannot be extended to a regular solution on $(0,\widetilde{T})$  for any $\widetilde{T}>T$, i.e. $T$ is a finite `blow-up time', then
$${\|u(t_{0})\|}_{L^{\infty}(\mathbb{R}^3)}\geq M{(T-t_{0})}^{-\frac{1}{2}}$$
for all $t_{0}\in(0,T)$. Leray also indicates in \cite{leray1934} how one may use various estimates for his solutions to extend this $L^\infty$ result to $L^p$ for any $p>3$, namely that for each $p>3$, there exists some\footnote{Leray's constant is $M_p:=M(1-\tfrac 3p)$.} $M_p>0$ such that
\begin{equation}\label{lebesgueburate}
{\|u(t_{0})\|}_{L^{p}(\mathbb{R}^3)}\geq M_p{(T-t_{0})}^{-\frac 12(1-\frac{3}{p})}
\end{equation}
for all $t_{0}\in(0,T)$ if $T$ is a time of finite blow-up (see, e.g., \cite{giga1986,robsad} for proofs, or \cite{ozanski2017} for the full details following Leray's outline).

In Theorem \ref{intro-leray} below, we extend\footnote{Note that our semigroup and uniqueness results  (Lemma \ref{semigrouplem} and Theorem \ref{uniqthm}) do not require the assumption that $f\in L^2(\mathbb{R}^n)$.  In fact, with such an additional assumption on $f$, we prove a `weak-strong uniqueness' result in \cite{davies2021} which (along with various estimates established here) implies that that any `Leray-Hopf' {\em weak} solution would coincide, on the common interval of existence, with the solution we construct here. On the other hand, Leray's solutions satisfying \eqref{lebesgueburate} are in fact (strong) mild solutions, i.e. solutions of \eqref{integral-equation}, which is precisely the type (see also Theorem \ref{intro-regularity} below along with the remarks which follow it) of solution we construct here, so that Theorem \ref{intro-leray} can be seen as a true extension of \eqref{lebesgueburate} established by Leray in \cite{leray1934}.} (\ref{lebesgueburate}) from the Lebesgue  setting $L^p$ to the more general Lorentz settings $L^{p,q}$ (see (\ref{lorstarnorm}) - (\ref{basic-relation-2}) below), and in fact on $\mathbb{R}^n$ more generally, for $p>n$.  Note that Theorem \ref{intro-leray} implies (\ref{lebesgueburate}) for $q>p$ (see (\ref{basic-relation-1}) - (\ref{basic-relation-2})). We note that a partial  result in this direction was recently established  by Y. Wang, W. Wei and H. Yu in \cite[Theorem 1.4]{wangetal2021}, as a by-product of an extension of the so-called `epsilon regularity' criterion (see, e.g., \cite{caf}) from the Lebesgue to the Lorentz setting.  The result in \cite{wangetal2021} is a regularity criterion (rather than a blow-up estimate), which is in fact implied\footnote{Effectively, \cite[Theorem 1.4]{wangetal2021} requires (\ref{lrinfburate}) to fail for all $t_0\in (0,T)$ in order to extend the solution beyond $T$, whereas here we show that if (\ref{lrinfburate}) fails for even one $t_0\in (0,T)$ then we may extend the solution beyond $T$.} by Theorem \ref{intro-leray}; moreover our proof is more direct for this purpose.  By standard Sobolev embeddings, Leray's result (\ref{lebesgueburate}) implies similar results in a certain range of subcritical homogeneous Sobolev spaces. We  note here that there has also been much recent interest in extending the subcritical Sobolev range for such results (see for example \cite{benameur2010,benorf2019,chesza,cormon,cormonpin,hwz2020,lorzin,morrvz,rss}).

A key feature of Leray's results is that the quantity ${\|f\|}_{L^{\infty}(\mathbb{R}^n)}$ has physical dimensions ${(\text{length})}^{-1}={(\text{rescaled time})}^{-1/2}$, so multiplying ${\|f\|}_{L^{\infty}(\mathbb{R}^n)}$ by a positive power of $T$ yields a dimensionless quantity. More generally, the quantity ${\|f\|}_{L^{p}(\mathbb{R}^{n})}$ has physical dimensions ${(\text{length})}^{-\left(1-\frac{n}{p}\right)}$, where $1-\frac{n}{p}>0$ for $p>n$ (in which case $L^{p}$ is \textbf{\textit{subcritical}}), and $1-\frac{n}{p}=0$ for $p=n$ (in which case $L^{p}$ is \textbf{\textit{critical}}).\footnote{Note that Leray's setting of $L^2(\mathbb{R}^n)$ is {\bf {\em supercritical}} for $n\geq 3$.} Fabes, Jones and Riviere \cite{fabes1972} consider the problem \eqref{integral-equation} in subcritical $L^{p}$ spaces, with the condition $T^{\frac{1}{2}}{\|f\|}_{L^{\infty}(\mathbb{R}^{n})}<M$ generalised to $T^{\frac{1}{2}\left(1-\frac{n}{p}\right)}{\|f\|}_{L^{p}(\mathbb{R}^{n})}<M_p$. Kato \cite{kato1984} considers the critical case, using methods which provide further estimates for solutions in subcritical spaces. The broader purpose of this paper is to generalise these local well-posedness results in the sub-critical setting by replacing subcritical Lebesgue spaces with subcritical Lorentz spaces.  In \cite{davies2021}, we furthermore establish the equivalence of \eqref{integral-equation} with various notions of weak solutions to (\ref{initial-value-problem}) in the settings we consider below.

To place the integral equation \eqref{integral-equation} on precise foundations, and to abbreviate some of the notation, we make the following definitions. The kernels $\Phi$ and $\mathcal{T}$ are defined by
\begin{equation}\label{heatkerdef}
    \Phi(t,x) := \frac{1}{{(4\pi t)}^{n/2}}e^{-{|x|}^{2}/4t} = \frac{1}{{(2\pi)}^{n}}\int_{\mathbb{R}^{n}}e^{\mathrm{i}x\cdot\xi-t{|\xi|}^{2}}\,\mathrm{d}\xi \quad \text{for }(t,x)\in(0,\infty)\times\mathbb{R}^{n},
\end{equation}
\begin{equation}\label{oseenkerdef}
    \mathcal{T}_{ij}(t,x) := \frac{1}{{(2\pi)}^{n}}\int_{\mathbb{R}^{n}}e^{\mathrm{i}x\cdot\xi-t{|\xi|}^{2}}\left(\delta_{ij}-\frac{\xi_{i}\xi_{j}}{{|\xi|}^{2}}\right)\,\mathrm{d}\xi \quad \text{for }(t,x)\in(0,\infty)\times\mathbb{R}^{n}.
\end{equation}
For $(t,x)\in(0,\infty)\times\mathbb{R}^{n}$ and a measurable function $f:\mathbb{R}^{n}\rightarrow\mathbb{R}^{n}$, we write
\begin{equation}\label{heat-map}
    S_{i}[f](t,x) := \int_{\mathbb{R}^{n}}\Phi(t,x-y)f_{i}(y)\,\mathrm{d}y
\end{equation}
whenever the integral is defined. For $(t,x)\in(0,\infty)\times\mathbb{R}^{n}$ and a measurable function $w:(0,t)\times\mathbb{R}^{n}\rightarrow\mathbb{R}^{n\times n}$, we write
\begin{equation}
    A_{i}[w](t,x) := \int_{0}^{t}\int_{\mathbb{R}^{n}}\nabla_{k}\mathcal{T}_{ij}(t-s,x-y)w_{jk}(s,y)\,\mathrm{d}y\,\mathrm{d}s
\end{equation}
whenever the integral is defined. For $(t,x)\in(0,\infty)\times\mathbb{R}^{n}$ and measurable functions $u,v:(0,t)\times\mathbb{R}^{n}\rightarrow\mathbb{R}^{n}$, we write
\begin{equation}\label{bilinear-operator}
    B_{i}[u,v](t,x) := A_{i}[u\otimes v](t,x)
\end{equation}
whenever $A_{i}[u\otimes v](t,x)$ is defined. For $T\in(0,\infty]$ and a measurable function $f:\mathbb{R}^{n}\rightarrow\mathbb{R}^{n}$, we are interested in measurable functions $u:(0,T)\times\mathbb{R}^{n}\rightarrow\mathbb{R}^{n}$ which satisfy
\begin{equation}\label{fixed-point-problem}
    u(t,x) = S[f](t,x) - B[u,u](t,x)
\end{equation}
for (at least) almost every $(t,x)\in(0,T)\times\mathbb{R}^{n}$.

Following \cite{hunt1966}, for a measurable function $f$ on $\mathbb{R}^{n}$ (and writing $|A|$ to denote the Lebesgue measure of a Borel subset $A\subseteq\mathbb{R}^{n}$), we define the quantities
\begin{equation}
    \lambda_{f}(y) := |\{x\in\mathbb{R}^{n}\text{ : }|f(x)|>y\}| \quad \text{for }y\in(0,\infty),
\end{equation}
\begin{equation}
    f^{*}(t) := \sup\{y\in(0,\infty)\text{ : }\lambda_{f}(y)>t\} \quad \text{for }t\in(0,\infty),
\end{equation}
\begin{equation}\label{lorstarnorm}
    {\|f\|}_{L^{p,q}(\mathbb{R}^{n})}^{*} := \left\{\begin{array}{ll}{\left(\frac{q}{p}\int_{0}^{\infty}{\left[t^{1/p}f^{*}(t)\right]}^{q}\,\frac{\mathrm{d}t}{t}\right)}^{1/q} & p\in[1,\infty),\,q\in[1,\infty), \\ \sup_{t\in(0,\infty)}t^{1/p}f^{*}(t) & p\in[1,\infty],\,q=\infty, \end{array}\right.
\end{equation}
which satisfy
\begin{equation}\label{basic-relation-1}
    {\|f\|}_{L^{p,p}(\mathbb{R}^{n})}^{*} = {\|f\|}_{L^{p}(\mathbb{R}^{n})} \quad \text{for }p\in[1,\infty],
\end{equation}
\begin{equation}\label{basic-relation-2}
    {\|f\|}_{L^{p,q_{2}}(\mathbb{R}^{n})}^{*} \leq {\|f\|}_{L^{p,q_{1}}(\mathbb{R}^{n})}^{*} \quad \text{for }p\in[1,\infty) \text{ and }1\leq q_{1}\leq q_{2}\leq\infty,
\end{equation}
\begin{equation}\label{basic-relation-3}
    {\|f\|}_{L^{p_{1},1}(\mathbb{R}^{n})}^{*} \lesssim_{p_{1},p_{2}} {|\supp(f)|}^{\frac{1}{p_{1}}-\frac{1}{p_{2}}}{\|f\|}_{L^{p_{2},\infty}(\mathbb{R}^{n})}^{*} \quad \text{if }1\leq p_{1}<p_{2}\leq\infty\text{ and }|\supp(f)|<\infty,
\end{equation}
\begin{equation}\label{basic-relation-4}
    \left\{\begin{array}{l} {\|f\|}_{L^{p,1}(\mathbb{R}^{n})}^{*} \lesssim_{p,p_{0},p_{1}} {\left({\|f\|}_{L^{p_{0},\infty}(\mathbb{R}^{n})}^{*}\right)}^{1-\theta} {\left({\|f\|}_{L^{p_{1},\infty}(\mathbb{R}^{n})}^{*}\right)}^{\theta} \\ \text{for }1\leq p_{0}<p_{1}\leq\infty,\text{ }\theta\in(0,1) \text{ and }\frac{1}{p}=\frac{1-\theta}{p_{0}}+\frac{\theta}{p_{1}}. \end{array}\right.
\end{equation}
In the case $p=\infty$, we have ${\|f\|}_{L^{\infty,\infty}(\mathbb{R}^{n})}^{*}={\|f\|}_{L^{\infty}(\mathbb{R}^{n})}=\esssup_{x\in\mathbb{R}^{n}}|f(x)|$. In situations where the actual supremum is desired, we find it convenient to write
\begin{equation}
    {\|f\|}_{L^{\overline{\infty},\infty}(\mathbb{R}^{n})}^{*} = {\|f\|}_{L^{\overline{\infty}}(\mathbb{R}^{n})} = \sup_{x\in\mathbb{R}^{n}}|f(x)|,
\end{equation}
so in general the exponent $p$ takes values in $[1,\overline{\infty}]:=[1,\infty]\cup\{\overline{\infty}\}$ when $q=\infty$. The \textbf{\textit{Lorentz space}} $L^{p,q}(\mathbb{R}^{n})$ consists of those measurable functions $f$ on $\mathbb{R}^{n}$ satisfying  ${\|f\|}_{L^{p,q}(\mathbb{R}^{n})}^{*}<\infty$, modulo the equivalence relation
\begin{equation}
    f =_{p} g \iff \left\{\begin{array}{ll} f =_{\mathrm{a.e.}} g & \text{if }p\in[1,\infty], \\ f =_{\mathrm{e.}} g & \text{if }p=\overline{\infty}, \end{array}\right.
\end{equation}
where $=_{\mathrm{a.e.}}$ and $=_{\mathrm{e.}}$ denote pointwise equality for (almost) every $x\in\mathbb{R}^{n}$. When we consider the intersection $L^{p,q}(\mathbb{R}^{n})\cap L^{\overline{\infty}}(\mathbb{R}^{n})$, we take the more refined equivalence relation $=_{\mathrm{e.}}$. The Lorentz space $L^{p,q}(\mathbb{R}^{n})$ has the same scaling properties as the Lebesgue space $L^{p}(\mathbb{R}^{n})$, so $L^{p,q}(\mathbb{R}^{n})$ is subcritical if and only if $p>n$.

Generalising the ideas of \cite{fabes1972} and \cite{kato1984}, we define the following \textbf{\textit{path spaces}} for $T\in(0,\infty]$, $\sigma\in\mathbb{R}$, $p\in[1,\overline{\infty}]$, and $q,\alpha,\beta\in[1,\infty]$ satisfying $(p\in\{\infty,\overline{\infty}\}\Rightarrow q=\infty)$ and $(\alpha=\infty\Rightarrow\beta=\infty)$. We define $\mathcal{L}_{p,q}^{\alpha,\beta}(T)$ to be the space of measurable functions $u$ on $(0,T)\times\mathbb{R}^{n}$ satisfying
\begin{equation}
    {\|u\|}_{\mathcal{L}_{p,q}^{\alpha,\beta}(T)}^{*} := {\left\|t\mapsto\mathbf{1}_{(0,T)}(t){\|u(t)\|}_{L^{p,q}(\mathbb{R}^{n})}^{*}\right\|}_{L^{\alpha,\beta}(\mathbb{R})}^{*}<\infty,
\end{equation}
modulo the equivalence relation
\begin{equation}
    u =_{p}^{\mathrm{a.e.}} v \iff u(t) =_{p} v(t) \text{ for almost every }t\in(0,T).
\end{equation}
We define $\mathcal{K}_{p,q}^{\sigma}(T)$ to be the space of measurable functions $u$ on $(0,T)\times\mathbb{R}^{n}$ satisfying
\begin{equation}
    {\|u\|}_{\mathcal{K}_{p,q}^{\sigma}(T)}^{*} := \esssup_{t\in(0,T)}t^{-\sigma/2}{\|u(t)\|}_{L^{p,q}(\mathbb{R}^{n})}^{*}<\infty,
\end{equation}
modulo the equivalence relation $=_{p}^{\mathrm{a.e.}}$. We define $\mathcal{J}_{p,q}^{\sigma}(T)$ to be the space of measurable functions $u$ on $(0,T)\times\mathbb{R}^{n}$ satisfying
\begin{equation}
    {\|u\|}_{\mathcal{J}_{p,q}^{\sigma}(T)}^{*} := \sup_{t\in(0,T)}t^{-\sigma/2}{\|u(t)\|}_{L^{p,q}(\mathbb{R}^{n})}^{*}<\infty,
\end{equation}
modulo the equivalence relation
\begin{equation}
    u =_{p}^{\mathrm{e.}} v \iff u(t) =_{p} v(t) \text{ for all }t\in(0,T).
\end{equation}
Let $\mathcal{L}_{p,q}^{\alpha,\beta}(T_{-}):=\cap_{T'\in(0,T)}\mathcal{L}_{p,q}^{\alpha,\beta}(T')$, $\mathcal{K}_{p,q}^{\sigma}(T_{-}):=\cap_{T'\in(0,T)}\mathcal{K}_{p,q}^{\sigma}(T')$ and $\mathcal{J}_{p,q}^{\sigma}(T_{-}):=\cap_{T'\in(0,T)}\mathcal{J}_{p,q}^{\sigma}(T')$. When considering an intersection of path spaces, we always take the more refined equivalence relation. Inequalities \eqref{basic-relation-2}-\eqref{basic-relation-4} (along with \eqref{prodineqlor} below) establish various relations concerning the quantities ${\|u\|}_{\mathcal{L}_{p,q}^{\alpha,\beta}(T)}^{*}$, ${\|u\|}_{\mathcal{K}_{p,q}^{\sigma}(T)}^{*}$ and ${\|u\|}_{\mathcal{J}_{p,q}^{\sigma}(T)}^{*}$, including
\begin{equation}\label{basic-inclusion}
    \left\{\begin{array}{l}
        {\|u\|}_{\mathcal{L}_{p,q}^{\alpha_{1},1}(T)}^{*} \lesssim_{\alpha_{1},\alpha_{2}} T^{\frac{1}{\alpha_{1}}-\frac{1}{\alpha_{2}}}{\|u\|}_{\mathcal{L}_{p,q}^{\alpha_{2},\infty}(T)}^{*} \leq T^{\frac{1}{\alpha_{1}}-\frac{1}{\alpha_{2}}}{\|u\|}_{\mathcal{K}_{p,q}^{-2/\alpha_{2}}(T)}^{*} = T^{\frac{1}{\alpha_{1}}-\frac{1}{\alpha_{2}}}{\|u\|}_{\mathcal{J}_{p,q}^{-2/\alpha_{2}}(T)}^{*} \\
        \text{for }1\leq\alpha_{1}<\alpha_{2}\leq\infty\text{ and }T\in(0,\infty).
    \end{array}\right.
\end{equation}

Our analysis of the equation \eqref{fixed-point-problem} is based the following estimates for the bilinear operator \eqref{bilinear-operator}. We use the convention that $\overline{p}=p$ for all $p\in[1,\infty)$.

\pagebreak

\begin{lemma}\label{intro-lemma}
    \begin{enumerate}[label=(\roman*)]
        \item Assume that $r\in(n,\infty]$, $\alpha\in(2,\infty)$, $\beta\in[1,\infty]$ and $\frac{2}{\alpha}+\frac{n}{r}\leq1$. Let $T\in(0,\infty)$, and let $u,v\in\mathcal{L}_{r,\infty}^{\alpha,\beta}(T)$. Then $B[u,v](t,x)$ is defined for almost every $(t,x)\in(0,T)\times\mathbb{R}^{n}$; in the case $r=\infty$, for almost every $t\in(0,T)$ we have that $B[u,v](t,x)$ is defined for all $x\in\mathbb{R}^{n}$. Moreover, we have the estimate
        \begin{equation}
            {\|B[u,v]\|}_{\mathcal{L}_{\overline{r},\infty}^{\alpha,\beta}(T)}^{*} \lesssim_{n,r,\alpha,\beta} T^{\frac{1}{2}\left(1-\frac{2}{\alpha}-\frac{n}{r}\right)}{\|u\|}_{\mathcal{L}_{r,\infty}^{\alpha,\beta}(T)}^{*}{\|v\|}_{\mathcal{L}_{r,\infty}^{\alpha,\beta}(T)}^{*}.
        \end{equation}
        \item Assume that $r\in(n,\infty]$, $\sigma\in(-1,\infty)$ and $-\sigma+\frac{n}{r}\leq1$. Let $T\in(0,\infty)$, and let $u,v\in\mathcal{K}_{r,\infty}^{\sigma}(T)$. Then for all $t\in(0,T)$ we have that $B[u,v](t,x)$ is defined for almost every $x\in\mathbb{R}^{n}$; in the case $r=\infty$, we have that $B[u,v](t,x)$ is defined for all $(t,x)\in(0,T)\times\mathbb{R}^{n}$. Moreover, we have the estimate
        \begin{equation}\label{intro-bilinear-estimate}
            {\|B[u,v]\|}_{\mathcal{J}_{\overline{r},\infty}^{\sigma}(T)}^{*} \lesssim_{n,r,\sigma} T^{\frac{1}{2}\left(1+\sigma-\frac{n}{r}\right)}{\|u\|}_{\mathcal{K}_{r,\infty}^{\sigma}(T)}^{*}{\|v\|}_{\mathcal{K}_{r,\infty}^{\sigma}(T)}^{*}.
        \end{equation}
    \end{enumerate}
\end{lemma}
If $p\in[1,\infty]$ and $f\in L^{p,\infty}(\mathbb{R}^{n})$, then we can show that $S[f](t,x)$ is defined for all $(t,x)\in(0,\infty)\times\mathbb{R}^{n}$. The uniqueness result of \cite{fabes1972} has the following generalisation. We use the convention that $\frac{1}{r}=0$ whenever $r\in\{\infty,\overline{\infty}\}$.
\begin{theorem}\label{intro-uniqueness}
    Assume that $p\in(1,\infty]$, $r\in(n,\overline{\infty}]$, $\alpha\in(2,\infty)$, $\beta\in[1,\infty)$, and $\frac{2}{\alpha}+\frac{n}{r}\leq1$. If $T\in(0,\infty]$, $f\in L^{p,\infty}(\mathbb{R}^{n})$, and $u,v\in\mathcal{L}_{r,\infty}^{\alpha,\beta}(T_{-})$ satisfy $S[f] =_{r}^{\mathrm{a.e.}} u+B[u,u] =_{r}^{\mathrm{a.e.}} v+B[v,v]$, then $u =_{r}^{\mathrm{a.e.}} v$.
\end{theorem}
Estimates \eqref{basic-inclusion} and \eqref{intro-bilinear-estimate} allow us to deduce further uniqueness results.
\begin{corollary}\label{intro-uniqueness-corollary}
    Assume that $p\in(1,\infty]$, $r\in(n,\overline{\infty}]$, $\sigma\in\mathbb{R}$, and $-\sigma+\frac{n}{r}<1$. Let $T\in(0,\infty]$ and $f\in L^{p,\infty}(\mathbb{R}^{n})$.
    \begin{enumerate}[label=(\roman*)]
        \item For any $\alpha\in(2,\infty)$ satisfying $-\sigma+\frac{n}{r}<\frac{2}{\alpha}+\frac{n}{r}\leq1$, and any $\beta\in[1,\infty)$, if $u\in\mathcal{K}_{r,\infty}^{\sigma}(T_{-})$ and $v\in\mathcal{L}_{r,\infty}^{\alpha,\beta}(T_{-})$ satisfy $S[f] =_{r}^{\mathrm{a.e.}} u+B[u,u] =_{r}^{\mathrm{a.e.}} v+B[v,v]$, then $u =_{r}^{\mathrm{a.e.}} v$.
        \item If $u,v\in\mathcal{J}_{r,\infty}^{\sigma}(T_{-})$ satisfy $S[f] =_{r}^{\mathrm{e.}} u+B[u,u] =_{r}^{\mathrm{e.}} v+B[v,v]$, then $u =_{r}^{\mathrm{e.}} v$.
    \end{enumerate}
\end{corollary}
\begin{proof}
    Part (i) follows immediately from \eqref{basic-inclusion} and Theorem \ref{intro-uniqueness}. For part (ii), we have $u =_{r}^{\mathrm{a.e.}} v$ by \eqref{basic-inclusion} and Theorem \ref{intro-uniqueness}, so it remains to show that $u =_{r}^{\mathrm{e.}} v$; by \eqref{intro-bilinear-estimate}, the $=_{r}^{\mathrm{e.}}$-equivalence class of $B[u,u]$ is determined by the $=_{r}^{\mathrm{a.e.}}$-equivalence class of $u$, so from $u+B[u,u] =_{r}^{\mathrm{e.}} v+B[v,v]$ and $u =_{r}^{\mathrm{a.e.}} v$ we deduce that $u =_{r}^{\mathrm{e.}} v$.
\end{proof}
The existence result of \cite{kato1984} has the following generalisation in subcritical Lorentz spaces. We write $r'\in[1,\infty]$ to denote the H\"{o}lder conjugate of $r\in[1,\infty]$.  In what follows, we will denote
\begin{equation}
    u =_{\mathrm{e.}}^{\mathrm{e.}} v \iff u(t) =_{\mathrm{e.}} v(t) \text{ for all }t\in(0,T) \iff u\equiv v \ \textrm{on} \ (0,T)\times \mathbb{R}^n.
\end{equation}
\begin{theorem}\label{intro-existence}
    For every $r\in(n,\infty]$, there exists a positive constant $\eta_{n,r}\lesssim_{n}r'\int_{0}^{1}{(1-s)}^{-\frac{1}{2}}s^{-\frac{n}{r}}\,\mathrm{d}s$ such that the following holds. For $p\in(1,\infty]$ and $q\in[1,\infty]$ satisfying $(p=\infty\Rightarrow q=\infty)$, if $T\in(0,\infty)$, $f\in L^{p,q}(\mathbb{R}^{n})\cap L^{r,\infty}(\mathbb{R}^{n})$, and $4\eta_{n,r}T^{\frac{1}{2}\left(1-\frac{n}{r}\right)}{\|f\|}_{L^{r,\infty}(\mathbb{R}^{n})}^{*}<1$, then there exists $u\in\mathcal{J}_{p,q}^{0}(T)\cap\mathcal{J}_{\overline{\infty},\infty}^{-n/r}(T)$ satisfying $u =_{\mathrm{e.}}^{\mathrm{e.}} S[f] - B[u,u]$.
\end{theorem}
Leray's blowup estimate then has the following generalisation.
\begin{theorem}\label{intro-leray}
    Assume that $p\in(1,\infty]$, $q\in[1,\infty]$, $(p=\infty\Rightarrow q=\infty)$ and $r_{0}\in(n,\infty]$. Let $f\in L^{p,q}(\mathbb{R}^{n})\cap L^{r_{0},\infty}(\mathbb{R}^{n})$, and let $T\in(0,\infty]$ be the maximal time for which there exists $u\in\mathcal{J}_{p,q}^{0}(T_{-})\cap\mathcal{J}_{\overline{\infty},\infty}^{-n/r_{0}}(T_{-})$ satisfying $u =_{\mathrm{e.}}^{\mathrm{e.}} S[f] - B[u,u]$. If $T<\infty$, then the unique solution $u\in\mathcal{J}_{p,q}^{0}(T_{-})\cap\mathcal{J}_{\overline{\infty},\infty}^{-n/r_{0}}(T_{-})$ satisfies
    \begin{equation}\label{lrinfburate}
        {\|u(t_{0})\|}_{L^{r,\infty}(\mathbb{R}^{n})}^{*} \geq \frac{1}{4\eta_{n,r}{(T-t_{0})}^{\frac{1}{2}\left(1-\frac{n}{r}\right)}} \quad \text{for all }r\in(n,\infty]\text{ and }t_{0}\in(0,T),
    \end{equation}
    where we use the convention that ${\|u(t_{0})\|}_{L^{r,\infty}(\mathbb{R}^{n})}^{*}=\infty$ if $u(t_{0})\notin L^{r,\infty}(\mathbb{R}^{n})$.  Note that, in view of (\ref{basic-relation-2}), (\ref{lrinfburate}) remains true with $L^{r,\infty}$ replaced  by $L^{r,q}$ for any $q\in [1,\infty]$.
\end{theorem}
\noindent

Inequality \eqref{basic-relation-4} allows us to establish further estimates on solutions. In particular, we obtain the following properties of solutions.

\pagebreak

\begin{corollary}\label{intro-characterisation}
    Given $f\in\cup_{p\in(n,\infty]}L^{p,\infty}(\mathbb{R}^{n})$, let $T_{*}=T_{*}[f]\in(0,\infty]$ be the maximal time for which there exists a solution $u_{*}\in\cup_{\sigma\in(-1,\infty)}\mathcal{J}_{\overline{\infty},\infty}^{\sigma}({T_{*}}_{-})$ to the equation $u_{*}=_{\mathrm{e.}}^{\mathrm{e.}}S[f]-B[u_{*},u_{*}]$, and write $u_{*}=u_{*}[f]$ to denote the unique solution $u_{*}\in\cup_{\sigma\in(-1,\infty)}\mathcal{J}_{\overline{\infty},\infty}^{\sigma}({T_{*}}_{-})$.
    \begin{enumerate}[label=(\roman*)]
        \item If $T_{*}<\infty$, then
        \begin{equation}\label{intro-blowup}
            {\|u_{*}(t_{0})\|}_{L^{r,\infty}(\mathbb{R}^{n})}^{*} \geq \frac{1}{4\eta_{n,r}{(T_{*}-t_{0})}^{\frac{1}{2}\left(1-\frac{n}{r}\right)}} \quad \text{for all }r\in(n,\infty]\text{ and }t_{0}\in(0,T_{*}).
        \end{equation}
        \item If $r_{0}\in(n,\infty]$ and $f\in L^{r_{0},\infty}(\mathbb{R}^{n})$, then
        \begin{equation}\label{intro-estimates}
            u_{*} \in \mathcal{J}_{r_{0},\infty}^{0}({T_{*}}_{-}) \cap \mathcal{J}_{\overline{\infty},\infty}^{-n/r_{0}}({T_{*}}_{-}) \cap \bigcap_{r\in(r_{0},\infty)}\mathcal{J}_{r,1}^{\frac{n}{r}-\frac{n}{r_{0}}}({T_{*}}_{-}).
        \end{equation}
        Moreover, if $r\in[r_{0},\overline{\infty}]$, $\alpha\in(2,\infty)$, $\beta\in[1,\infty)$, $\frac{n}{r_{0}}<\frac{2}{\alpha}+\frac{n}{r}\leq1$, $T\in(0,\infty]$ and $v\in\mathcal{L}_{r,\infty}^{\alpha,\beta}(T_{-})$ satisfies $v=_{r}^{\mathrm{a.e.}}S[f]-B[v,v]$, then $T\leq T_{*}$ and $u_{*}=_{r}^{\mathrm{a.e.}}v$ on $(0,T)$.
        \item If $r_{0}\in(n,\infty]$, $p\in(1,r_{0}]$, $q\in[1,\infty]$, $(p=\infty\Rightarrow q=\infty)$ and $f\in L^{p,q}(\mathbb{R}^{n})\cap L^{r_{0},\infty}(\mathbb{R}^{n})$, then
        \begin{equation}\label{intro-further-estimates}
            u_{*} \in \mathcal{J}_{p,q}^{0}({T_{*}}_{-}) \cap \bigcap_{r\in(p,r_{0})}\mathcal{J}_{r,1}^{0}({T_{*}}_{-}).
        \end{equation}
    \end{enumerate}
\end{corollary}
\begin{proof}
    We start by proving \eqref{intro-blowup}, \eqref{intro-estimates} and \eqref{intro-further-estimates}. Suppose that $r_{0}\in(n,\infty]$, $p\in(1,r_{0}]$, $q\in[1,\infty]$, $(p=\infty\Rightarrow q=\infty)$ and $f\in L^{p,q}(\mathbb{R}^{n})\cap L^{r_{0},\infty}(\mathbb{R}^{n})$, and let $T\in(0,\infty]$ be the maximal time for which there exists $v\in\mathcal{J}_{p,q}^{0}(T_{-})\cap\mathcal{J}_{\overline{\infty},\infty}^{-n/r_{0}}(T_{-})$ satisfying $v=_{\mathrm{e.}}^{\mathrm{e.}}S[f]-B[v,v]$. Then $T\leq T_{*}$, and by Corollary \ref{intro-uniqueness-corollary} we have $u_{*}=_{\mathrm{e.}}^{\mathrm{e.}}v$ on $(0,T)$. If $T<\infty$, then by Theorem \ref{intro-leray} we have $4\eta_{n,r}{(T-t_{0})}^{\frac{1}{2}\left(1-\frac{n}{r}\right)}{\|u_{*}(t_{0})\|}_{L^{r,\infty}(\mathbb{R}^{n})}^{*}\geq1$ for all $r\in(n,\infty]$ and $t_{0}\in(0,T)$. In the case $r=\infty$, this rules out the possibility that $T<T_{*}$. Therefore $T=T_{*}$, and \eqref{intro-blowup} holds if $T_{*}<\infty$. We deduce that $u_{*}\in\mathcal{J}_{p,q}^{0}({T_{*}}_{-})\cap\mathcal{J}_{\overline{\infty},\infty}^{-n/r_{0}}({T_{*}}_{-})$, so in the special case $p=r_{0}$ and $q=\infty$ we obtain $u_{*}\in\mathcal{J}_{r_{0},\infty}^{0}({T_{*}}_{-})$. Inequality \eqref{basic-relation-4} then implies \eqref{intro-estimates} and \eqref{intro-further-estimates}.

    We now prove the ``moreover'' statement from part (ii). Suppose that $r_{0}\in(n,\infty]$, $f\in L^{r_{0},\infty}(\mathbb{R}^{n})$, $r\in[r_{0},\overline{\infty}]$, $\alpha\in(2,\infty)$, $\beta\in[1,\infty)$, $\frac{n}{r_{0}}<\frac{2}{\alpha}+\frac{n}{r}\leq1$, $T\in(0,\infty]$, $v\in\mathcal{L}_{r,\infty}^{\alpha,\beta}(T_{-})$, and $v=_{r}^{\mathrm{a.e.}}S[f]-B[v,v]$. By \eqref{intro-estimates} we have $u_{*}\in\mathcal{J}_{r,\infty}^{\frac{n}{r}-\frac{n}{r_{0}}}({T_{*}}_{-})$, so by Corollary \ref{intro-uniqueness-corollary} we have $u_{*}=_{r}^{\mathrm{a.e.}}v$ on the common interval of existence. The blowup estimate \eqref{intro-blowup} rules out the possibility that $T>T_{*}$.
\end{proof}
The arguments of \cite{ozanski2017} generalise to provide the following regularity properties of solutions.
\begin{theorem}\label{intro-regularity}
    Assume that $p\in(1,\infty]$, $q\in[1,\infty]$ and $(p=\infty\Rightarrow q=\infty)$. Let $f\in L^{p,q}(\mathbb{R}^{n})\cap\left(\cup_{r\in(n,\infty]}L^{r,\infty}(\mathbb{R}^{n})\right)$, and let $T_{*}=T_{*}[f]$ and $u_{*}=u_{*}[f]$ be as described in Corollary \ref{intro-characterisation}. Then $u_{*}$ defines a continuous function $u_{*}:(0,T_{*})\rightarrow L^{p,q}(\mathbb{R}^{n})\cap L^{\overline{\infty}}(\mathbb{R}^{n})$, and for each $\alpha\in(0,1)$ the H\"{o}lder seminorm ${[u_{*}(t)]}_{C^{\alpha}(\mathbb{R}^{n})}$ is locally bounded in $t\in(0,T_{*})$. If $q<\infty$, then ${\|u_{*}(t)-f\|}_{L^{p,q}(\mathbb{R}^{n})}^{*}\rightarrow0$ as $t\rightarrow0$.
\end{theorem}
Under the additional assumption that $f$ is weakly divergence-free, one can adapt the arguments of \cite{ozanski2017} to show that $u_{*}$ is divergence-free and smooth in the spatial variables and if, moreover, $f\in L^2(\mathbb{R}^n)$ then it is smooth in both space and time, but we will not do so here.\\

The rest of this paper is organised as follows. In section 3 we give a brief overview of Lorentz spaces. In section 4 we establish properties of the kernels $\Phi$ and $\mathcal{T}$. In section 5 we establish properties of the integral operators $S$, $A$ and $B$. In section 6 we prove Theorem \ref{intro-uniqueness}. In section 7 we prove Theorems \ref{intro-existence} and \ref{intro-leray}. In section 8 we prove Theorem \ref{intro-regularity}.
\section{Lorentz spaces}
For a measurable function $f$ on $\mathbb{R}^{n}$, we define the \textbf{\textit{distribution function}}
\begin{equation}
    \lambda_{f}(y) := |\{x\in\mathbb{R}^{n}\text{ : }|f(x)|>y\}| \quad \text{for }y\in(0,\infty)
\end{equation}
and the \textbf{\textit{decreasing rearrangement}}
\begin{equation}
    f^{*}(t) := \sup\{y\in(0,\infty)\text{ : }\lambda_{f}(y)>t\} \quad \text{for }t\in(0,\infty),
\end{equation}
which satisfies the \textbf{\textit{rearrangement inequality}} (\cite{hunt1966}, inequality (1.9)):
\begin{equation}
    \int_{\mathbb{R}^{n}}|f(x)g(x)|\,\mathrm{d}x \leq \int_{0}^{\infty}f^{*}(t)g^{*}(t)\,\mathrm{d}t.
\end{equation}
The \textbf{\textit{Lorentz quasinorms}} are defined by
\begin{equation}
    {\|f\|}_{L^{p,q}(\mathbb{R}^{n})}^{*} := \left\{\begin{array}{ll}{\left(\frac{q}{p}\int_{0}^{\infty}{\left[t^{1/p}f^{*}(t)\right]}^{q}\,\frac{\mathrm{d}t}{t}\right)}^{1/q} & p\in[1,\infty),\,q\in[1,\infty), \\ \sup_{t\in(0,\infty)}t^{1/p}f^{*}(t) & p\in[1,\infty],\,q=\infty.\end{array}\right.
\end{equation}
The Lorentz quasinorms are a generalisation of the Lebesgue norms (\cite{hunt1966}, p.\ 253):
\begin{equation}
    {\|f\|}_{L^{p,p}(\mathbb{R}^{n})}^{*} = {\|f\|}_{L^{p}(\mathbb{R}^{n})} \quad \text{for }p\in[1,\infty].
\end{equation}
We have the \textbf{\textit{Lorentz embedding}} (\cite{hunt1966}, inequality (1.8)):
\begin{equation}
    {\|f\|}_{L^{p,q_{2}}(\mathbb{R}^{n})}^{*} \leq {\|f\|}_{L^{p,q_{1}}(\mathbb{R}^{n})}^{*} \quad \text{for }p\in[1,\infty)\text{ and }1\leq q_{1}\leq q_{2}\leq\infty.
\end{equation}
If $f$ is supported within a set of measure $T$, then $f^{*}$ is supported within $(0,T)$, in which case we have
\begin{equation}\label{local-embedding}
    {\|f\|}_{L^{p_{1},1}(\mathbb{R}^{n})}^{*} \lesssim_{p_{1},p_{2}} T^{\frac{1}{p_{1}}-\frac{1}{p_{2}}}{\|f\|}_{L^{p_{2},\infty}(\mathbb{R}^{n})}^{*} \quad \text{for }1\leq p_{1}<p_{2}\leq\infty.
\end{equation}
Simple functions (supported on sets of finite measure) are dense in $L^{p,q}(\mathbb{R}^{n})$ for $p,q\in[1,\infty)$ (\cite{hunt1966}, statement (2.4)). Arguing along the lines of (\cite{mouhot2017}, Theorem 1.26), it follows for $p,q\in[1,\infty)$ that $L^{p,q}(\mathbb{R}^{n})$ is separable and contains $C_{c}^{0}(\mathbb{R}^{n})$ as a dense subset. For $p,q\in[1,\infty)$ we deduce the \textbf{\textit{continuity of translation}} result ${\|f(\cdot-h)-f(\cdot)\|}_{L^{p,q}(\mathbb{R}^{n})}^{*}\overset{h\rightarrow0}{\rightarrow}0$ (this is true for $f\in C_{c}^{0}(\mathbb{R}^{n})$ by dominated convergence, so is true for $f\in L^{p,q}(\mathbb{R}^{n})$ by approximation).

\textbf{\textit{Hardy's inequalities}} (\cite{hunt1966}, p.\ 256) state that for $p\in(0,\infty)$ and $q\in[1,\infty)$ we have
\begin{equation}
\begin{aligned}
    {\left(\int_{0}^{\infty}{\left[\int_{0}^{t}t^{-\frac{1}{p}}\phi(s)\,\frac{\mathrm{d}s}{s}\right]}^{q}\,\frac{\mathrm{d}t}{t}\right)}^{\frac{1}{q}} &\leq p{\left(\int_{0}^{\infty}{\left[s^{-\frac{1}{p}}\phi(s)\right]}^{q}\,\frac{\mathrm{d}s}{s}\right)}^{\frac{1}{q}}, \\
    {\left(\int_{0}^{\infty}{\left[\int_{t}^{\infty}t^{\frac{1}{p}}\phi(s)\,\frac{\mathrm{d}s}{s}\right]}^{q}\,\frac{\mathrm{d}t}{t}\right)}^{\frac{1}{q}} &\leq p{\left(\int_{0}^{\infty}{\left[s^{\frac{1}{p}}\phi(s)\right]}^{q}\,\frac{\mathrm{d}s}{s}\right)}^{\frac{1}{q}}.
\end{aligned}
\end{equation}
Hardy's inequalities can be used to prove the following theorem (\cite{hunt1966}, p.\ 264) on the \textbf{\textit{interpolation of operators}}: if $p_{0}\neq p_{1}$, $r_{0}\neq r_{1}$, and $T:L^{p_{0},q_{0}}(\mathbb{R}^{n})+L^{p_{1},q_{1}}(\mathbb{R}^{n})\rightarrow L^{r_{0},s_{0}}(\mathbb{R}^{n})+L^{r_{1},s_{1}}(\mathbb{R}^{n})$ is a function satisfying
\begin{equation}
    |T(f+g)| \leq K(|Tf|+|Tg|) \text{ a.e., } \quad {\|Tf\|}_{L^{r_{i},s_{i}}(\mathbb{R}^{n})}^{*} \leq B_{i}{\|f\|}_{L^{p_{i},q_{i}}(\mathbb{R}^{n})}^{*} \text{ for }i\in\{0,1\},
\end{equation}
then for $\theta\in(0,1)$, $\frac{1}{p_{\theta}}=\frac{1-\theta}{p_{0}}+\frac{\theta}{p_{1}}$ and $\frac{1}{r_{\theta}}=\frac{1-\theta}{r_{0}}+\frac{\theta}{r_{1}}$ we have $L^{p_{\theta},\infty}(\mathbb{R}^{n})\subseteq L^{p_{0},q_{0}}(\mathbb{R}^{n})+L^{p_{1},q_{1}}(\mathbb{R}^{n})$ and
\begin{equation}
    {\|Tf\|}_{L^{r_{\theta},s}(\mathbb{R}^{n})}^{*} \leq B_{\theta}{\|f\|}_{L^{p_{\theta},q}(\mathbb{R}^{n})}^{*} \quad \text{for }q\leq s,
\end{equation}
where $B_{\theta}$ depends on $\theta$, the Lorentz indices and the constants $K,B_{0},B_{1}$.

For a measurable function $f$ on $\mathbb{R}^{n}$, we define the \textbf{\textit{maximal function}}
\begin{equation}
    f^{**}(t) := \sup_{|A|\geq t}\frac{1}{|A|}\int_{A}|f(x)|\,\mathrm{d}x = \frac{1}{t}\int_{0}^{t}f^{*}(s)\,\mathrm{d}s \quad \text{for }t\in(0,\infty).
\end{equation}
To prove the equivalence of the two expressions for $f^{**}(t)$, it suffices to consider simple functions and apply monotone convergence. If $f$ is a simple function (supported on sets of finite measure), then the map
\begin{equation}
    \tau_{f}(x) := |\{z\in\mathbb{R}^{n}\text{ : }|f(z)|>|f(x)|\}|+|\{z\in\mathbb{R}^{n}\text{ : }|z|<|x|,|f(z)|=|f(x)|\}|
\end{equation}
defines a measure preserving transformation $\tau_{f}:(\mathbb{R}^{n},\mathrm{d}x)\rightarrow((0,\infty),\mathrm{d}s)$ satisfying $|f|=f^{*}\circ\tau_{f}$ almost everywhere, so
\begin{equation}
    \sup_{|A|\geq t}\frac{1}{|A|}\int_{A}|f(x)|\,\mathrm{d}x = \sup_{|B|\geq t}\frac{1}{|B|}\int_{B}f^{*}(s)\,\mathrm{d}s = \frac{1}{t}\int_{0}^{t}f^{*}(s)\,\mathrm{d}s.
\end{equation}
For $p\in(1,\infty]$ and $q\in[1,\infty]$ satisfying $(p=\infty\Rightarrow q=\infty)$, it follows from the expression $f^{**}(t)=\frac{1}{t}\int_{0}^{t}f^{*}(s)\,\mathrm{d}s$ and Hardy's inequality that
\begin{equation}
    f^{*}(t)\leq f^{**}(t), \quad {\|f\|}_{L^{p,q}(\mathbb{R}^{n})}^{*} \leq {\|f\|}_{L^{p,q}(\mathbb{R}^{n})} \leq p'{\|f\|}_{L^{p,q}(\mathbb{R}^{n})}^{*},
\end{equation}
where the \textbf{\textit{Lorentz norms}} are defined by
\begin{equation}
    {\|f\|}_{L^{p,q}(\mathbb{R}^{n})} := \left\{\begin{array}{ll}{\left(\frac{q}{p}\int_{0}^{\infty}{\left[t^{1/p}f^{**}(t)\right]}^{q}\,\frac{\mathrm{d}t}{t}\right)}^{1/q} & p\in(1,\infty),\,q\in[1,\infty), \\ \sup_{t\in(0,\infty)}t^{1/p}f^{**}(t) & p\in(1,\infty],\,q=\infty.\end{array}\right.
\end{equation}

If $(E,\mathcal{E})$ is a measurable space, and $f$ is a measurable function on $E\times\mathbb{R}^{n}$ satisfying $f(\theta)\in L^{p,q}(\mathbb{R}^{n})$ for all $\theta\in E$, then the expressions $\lambda_{f(\theta)}(y)$, ${[f(\theta)]}^{*}(t)$ and ${[f(\theta)]}^{**}(t)$ define measurable functions on $E\times(0,\infty)$, so the expressions ${\|f(\theta)\|}_{L^{p,q}(\mathbb{R}^{n})}^{*}$ and ${\|f(\theta)\|}_{L^{p,q}(\mathbb{R}^{n})}$ define measurable functions on $E$. If $\mu$ is a $\sigma$-finite measure on $(E,\mathcal{E})$, then by the expression $f^{**}(t) = \sup_{|A|\geq t}\frac{1}{|A|}\int_{A}|f(x)|\,\mathrm{d}x$ and Fubini's theorem we have ${\left[\int_{E}f(\theta,\cdot)\,\mathrm{d}\mu(\theta)\right]}^{**}(t)\leq\int_{E}{[f(\theta)]}^{**}(t)\,\mathrm{d}\mu(\theta)$, so by Minkowski's inequality for Lebesgue spaces we obtain \textbf{\textit{Minkowski's inequality for Lorentz spaces}}
\begin{equation}
    {\left\|\int_{E}f(\theta,\cdot)\,\mathrm{d}\mu(\theta)\right\|}_{L^{p,q}(\mathbb{R}^{n})} \leq \int_{E}{\|f(\theta)\|}_{L^{p,q}(\mathbb{R}^{n})}\,\mathrm{d}\mu(\theta).
\end{equation}
In the case where $E$ is countable and $\mu$ is the counting measure, we deduce that if $\sum_{m}{\|f_{m}\|}_{L^{p,q}(\mathbb{R}^{n})}<\infty$ then $\sum_{m}f_{m}$ converges almost everywhere and in $L^{p,q}(\mathbb{R}^{n})$. Therefore ${\|\cdot\|}_{L^{p,q}(\mathbb{R}^{n})}$ defines a Banach space norm on $L^{p,q}(\mathbb{R}^{n})$.

For measurable functions $f$ and $g$ on $\mathbb{R}^{n}$, we have the \textbf{\textit{product inequality}}
\begin{equation}\label{prodineqlor}
    {(fg)}^{*}(t) \leq f^{**}(t)g^{**}(t).
\end{equation}
The product inequality is proved by noting that ${(fg)}^{*}(t)={\left({(\sqrt{|fg|})}^{*}(t)\right)}^{2}\leq{\left({(\sqrt{|fg|})}^{**}(t)\right)}^{2}$, which we estimate using the rearrangement inequality and H\"{o}lder's inequality. The product inequality combines with H\"{o}lder's inequality to give estimates of the form ${\|fg\|}_{L^{p,q}(\mathbb{R}^{n})}^{*}\lesssim_{p_{0},p_{1},q_{0},q_{1}}{\|f\|}_{L^{p_{0},q_{0}}(\mathbb{R}^{n})}{\|g\|}_{L^{p_{1},q_{1}}(\mathbb{R}^{n})}$ for $\frac{1}{p}=\frac{1}{p_{0}}+\frac{1}{p_{1}}$ and $\frac{1}{q}=\frac{1}{q_{0}}+\frac{1}{q_{1}}$. These estimates are complemented by the trivial estimate \linebreak ${\|fg\|}_{L^{p,q}(\mathbb{R}^{n})}\leq{\|f\|}_{L^{p,q}(\mathbb{R}^{n})}{\|g\|}_{L^{\infty}(\mathbb{R}^{n})}$.

We also have the \textbf{\textit{convolution inequality}}
\begin{equation}
    {(f*g)}^{**}(t) \leq {\|f\|}_{L^{1}(\mathbb{R}^{n})}g^{**}(t),
\end{equation}
which is proved using Fubini's theorem, and implies the estimate ${\|f*g\|}_{L^{p,q}(\mathbb{R}^{n})}\leq{\|f\|}_{L^{1}(\mathbb{R}^{n})}{\|g\|}_{L^{p,q}(\mathbb{R}^{n})}$. Combining the convolution inequality ${\|f*g\|}_{L^{p,\infty}(\mathbb{R}^{n})}\leq{\|f\|}_{L^{1}(\mathbb{R}^{n})}{\|g\|}_{L^{p,\infty}(\mathbb{R}^{n})}$ with the rearrangement inequality ${\|f*g\|}_{L^{\overline{\infty}}(\mathbb{R}^{n})}\leq p'{\|f\|}_{L^{p',1}(\mathbb{R}^{n})}^{*}{\|g\|}_{L^{p,\infty}(\mathbb{R}^{n})}^{*}$ (where ${\|\cdot\|}_{L^{\overline{\infty}}(\mathbb{R}^{n})}$ denotes the supremum norm), we can use interpolation of operators to deduce estimates of the form ${\|f*g\|}_{L^{r,s}(\mathbb{R}^{n})}\lesssim_{p,q,s}{\|f\|}_{L^{q,s}(\mathbb{R}^{n})}{\|g\|}_{L^{p,\infty}(\mathbb{R}^{n})}$ for $1+\frac{1}{r}=\frac{1}{p}+\frac{1}{q}$.

For $p,q\in[1,\infty)$ satisfying $(p=1\Rightarrow q=1)$, the convolution inequality ${\|f*g\|}_{L^{p,q}(\mathbb{R}^{n})}^{*}\lesssim_{p,q}{\|f\|}_{L^{1}(\mathbb{R}^{n})}{\|g\|}_{L^{p,q}(\mathbb{R}^{n})}^{*}$ (proved for $p=q=1$ by Fubini's theorem) allows us to prove the \textbf{\textit{approximation of identity}} result ${\left\|\int_{\mathbb{R}^{n}}\epsilon^{-n}f(y/\epsilon)\left(g(\cdot-y)-g(\cdot)\right)\,\mathrm{d}y\right\|}_{L^{p,q}(\mathbb{R}^{n})}^{*}\overset{\epsilon\rightarrow0}{\rightarrow}0$ for $f\in L^{1}(\mathbb{R}^{n})$ and $g\in L^{p,q}(\mathbb{R}^{n})$ (this is true for $g\in C_{c}^{0}(\mathbb{R}^{n})$ by the substitution $z=y/\epsilon$ and dominated convergence, so is true for $g\in L^{p,q}(\mathbb{R}^{n})$ by approximation). The same result holds if $L^{p,q}(\mathbb{R}^{n})$ is replaced by the space $C_{b,u}^{0}(\mathbb{R}^{n})$ of bounded uniformly continuous functions (approximate $f$ in $L^{1}(\mathbb{R}^{n})$ by a compactly supported function, then apply continuity of translation in $C_{b,u}^{0}(\mathbb{R}^{n})$). Using density of $C_{c}^{0}(\mathbb{R}^{n})$, and approximating using a smooth, compactly supported mollifier, we deduce that $C_{c}^{\infty}(\mathbb{R}^{n})$ is dense in $L^{p,q}(\mathbb{R}^{n})$ for $p,q\in[1,\infty)$ satisfying $(p=1\Rightarrow q=1)$.

For $\frac{1}{p}=\frac{1-\theta}{p_{0}}+\frac{\theta}{p_{1}}$, $1<p_{0}<p_{1}\leq\infty$ and $\theta\in(0,1)$, we have the \textbf{\textit{interpolation of norms}}
\begin{equation}
    {\|f\|}_{L^{p,\infty}(\mathbb{R}^{n})}\leq{\|f\|}_{L^{p_{0},\infty}(\mathbb{R}^{n})}^{1-\theta}{\|f\|}_{L^{p_{1},\infty}(\mathbb{R}^{n})}^{\theta}, \quad {\|f\|}_{L^{p,1}(\mathbb{R}^{n})}\leq\frac{2{\|f\|}_{L^{p_{0},\infty}(\mathbb{R}^{n})}^{1-\theta}{\|f\|}_{L^{p_{1},\infty}(\mathbb{R}^{n})}^{\theta}}{p\left(\frac{1}{p_{0}}-\frac{1}{p_{1}}\right)\theta^{1-\theta}{(1-\theta)}^{\theta}}.
\end{equation}
The first inequality follows immediately from the definition of Lorentz norms, while the second inequality is derived by writing ${\|f\|}_{L^{p,1}(\mathbb{R}^{n})}=\frac{1}{p}\left(\int_{0}^{t_{0}}+\int_{t_{0}}^{\infty}\right)t^{1/p}f^{**}(t)\,\frac{\mathrm{d}t}{t}$ with $t_{0}^{\frac{1}{p_{0}}-\frac{1}{p_{1}}}=\frac{(1-\theta){\|f\|}_{L^{p_{0},\infty}(\mathbb{R}^{n})}}{\theta{\|f\|}_{L^{p_{1},\infty}(\mathbb{R}^{n})}}$. The same inequalities hold with Lorentz norms replaced by Lorentz quasinorms (and the assumption $p_{0}>1$ replaced by $p_{0}\geq1$).
\section{Heat-type kernels}
For $t\in(0,\infty)$ and $x\in\mathbb{R}^{n}$ we define
\begin{equation}\label{kernel-definition}
    \Phi(t,x) := \frac{1}{{(2\pi)}^{n}}\int_{\mathbb{R}^{n}}e^{\mathrm{i}x\cdot\xi-t{|\xi|}^{2}}\,\mathrm{d}\xi, \quad \mathcal{T}_{ij}(t,x) := \frac{1}{{(2\pi)}^{n}}\int_{\mathbb{R}^{n}}e^{\mathrm{i}x\cdot\xi-t{|\xi|}^{2}}\left(\delta_{ij}-\frac{\xi_{i}\xi_{j}}{{|\xi|}^{2}}\right)\,\mathrm{d}\xi.
\end{equation}
The heat kernel $\Phi$ is just the inverse Fourier transform of the Gaussian $e^{-t{|\xi|}^{2}}$, for which we have the explicit formula
\begin{equation}\label{explicit-heat-kernel}
    \Phi(t,x) = \frac{1}{{(4\pi t)}^{n/2}}e^{-{|x|}^{2}/4t}.
\end{equation}
More generally, the heat kernel $\Phi$ and the Oseen kernel $\mathcal{T}$ are examples of the heat-type kernel
\begin{equation}\label{heat-type-kernel}
    K(t,x) := \frac{1}{{(2\pi)}^{n}}\int_{\mathbb{R}^{n}}e^{\mathrm{i}x\cdot\xi-t{|\xi|}^{2}}h(\xi)\,\mathrm{d}\xi \quad \text{for }t\in(0,\infty)\text{ and }x\in\mathbb{R}^{n},
\end{equation}
where $h$ is a smooth function on $\mathbb{R}^{n}\setminus\{0\}$ which is \textbf{\textit{homogeneous}} in the sense that $h(\lambda\xi)=h(\xi)$ for all $\xi\in\mathbb{R}^{n}\setminus\{0\}$ and $\lambda\in(0,\infty)$. Writing $D=-\mathrm{i}\nabla$, and differentiating (\ref{heat-type-kernel}) under the integral, we deduce that $K$ is smooth in the spatial variables, satisfying
\begin{equation}\label{heat-type-derivative}
    D^{\alpha}K(t,x) = \frac{1}{{(2\pi)}^{n}}\int_{\mathbb{R}^{n}}e^{\mathrm{i}x\cdot\xi-t{|\xi|}^{2}}\xi^{\alpha}h(\xi)\,\mathrm{d}\xi \quad \text{for }t\in(0,\infty)\text{ and }x\in\mathbb{R}^{n}
\end{equation}
for every multi-index $\alpha\in\mathbb{Z}_{\geq0}^{n}$. Using the representation (\ref{heat-type-derivative}), we obtain the following pointwise estimates.
\begin{lemma}
    For each $\alpha\in\mathbb{Z}_{\geq0}^{n}$, we have the pointwise estimate
    \begin{equation}\label{heat-type-pointwise}
        |D^{\alpha}K(t,x)| \lesssim_{\alpha,h,n} t^{-\frac{1}{2}(n+|\alpha|)}{\left(1+\frac{|x|}{\sqrt{t}}\right)}^{-(n+|\alpha|)} \quad \text{for }t\in(0,\infty)\text{ and }x\in\mathbb{R}^{n}.
    \end{equation}
\end{lemma}
\begin{proof}
    We follow the argument of (\cite{lemarierieusset2002}, Propostion 11.1). Let $\rho:\mathbb{R}^{n}\rightarrow[0,1]$ be a smooth function satisfying $\rho(\xi)=1$ for $|\xi|\leq1/2$, and $\rho(\xi)=0$ for $|\xi|\geq1$. Then $\sigma(\xi):=\rho(\xi/2)-\rho(\xi)$ defines a smooth function $\sigma:\mathbb{R}^{n}\rightarrow[0,1]$ supported on $\{\xi\in\mathbb{R}^{n}\mid1/2\leq|\xi|\leq2\}$, with the property that $\sum_{j\in\mathbb{Z}}\sigma(2^{-j}\xi)=1$ for all $\xi\in\mathbb{R}^{n}\setminus\{0\}$. By dominated convergence, we therefore have $D^{\alpha}K(t,x)=\sum_{j\in\mathbb{Z}}\Delta_{j}D^{\alpha}K(t,x)$, where we {\em define}
    \begin{equation}\label{littlewood-paley}
        \Delta_{j}D^{\alpha}K(t,x) := \frac{1}{{(2\pi)}^{n}}\int_{\mathbb{R}^{n}}e^{\mathrm{i}x\cdot\xi-t{|\xi|}^{2}}\xi^{\alpha}h(\xi)\sigma(2^{-j}\xi)\,\mathrm{d}\xi \quad \text{for }t\in(0,\infty)\text{ and }x\in\mathbb{R}^{n}.
    \end{equation}
    By change of variable in (\ref{littlewood-paley}), we have $\Delta_{j}D^{\alpha}K(t,x)=2^{j(n+|\alpha|)}\Delta_{0}D^{\alpha}K(2^{2j}t,2^{j}x)$. By the Leibniz rule, for $\beta\in\mathbb{Z}_{\geq0}^{n}$, $s>0$ and $\xi\in\mathbb{R}^{n}$ with $1/2\leq\xi\leq2$ we have $\left|D^{\beta}[\xi\mapsto e^{-s{|\xi|}^{2}}\xi^{\alpha}h(\xi)\sigma(\xi)](\xi)\right|\lesssim e^{-s{|\xi|}^{2}}{(1+s{|\xi|}^{2})}^{|\beta|}$, where the implied constant is independent of $s$, so ${\left\|D^{\beta}[\xi\mapsto e^{-s{|\xi|}^{2}}\xi^{\alpha}h(\xi)\sigma(\xi)]\right\|}_{L^{1}(\mathbb{R}^{n})}$ is bounded independently of $s$. Taking the definition (\ref{littlewood-paley}) of $\Delta_{0}D^{\alpha}K(s,y)$ and integrating by parts, we deduce for $\beta\in\mathbb{Z}_{\geq0}^{n}$ that $y^{\beta}\Delta_{0}D^{\alpha}K(s,y)$ is bounded independently of $s>0$ and $y\in\mathbb{R}^{n}$. For $N\in\mathbb{N}$, therefore ${(1+|y|)}^{N}\Delta_{0}D^{\alpha}K(s,y)$ is bounded independently of $s>0$ and $y\in\mathbb{R}^{n}$. Choosing $N>n+|\alpha|$, for $t>0$ and $x\in\mathbb{R}^{n}\setminus\{0\}$ we therefore have
    \begin{equation}
    \begin{aligned}
        |D^{\alpha}K(t,x)| &\leq \sum_{j\in\mathbb{Z}}2^{j(n+|\alpha|)}|\Delta_{0}D^{\alpha}K(2^{2j}t,2^{j}x)| \lesssim \sum_{j\in\mathbb{Z}}\frac{2^{j(n+|\alpha|)}}{{(1+2^{j}|x|)}^{N}} \\
        &\qquad \leq \sum_{2^{j}|x|\leq1}2^{j(n+|\alpha|)} + \sum_{2^{j}|x|>1}\frac{2^{j(n+|\alpha|-N)}}{{|x|}^{N}} \lesssim {|x|}^{-(n+|\alpha|)}.
    \end{aligned}
    \end{equation}
    By change of variable in (\ref{heat-type-derivative}), we have $D^{\alpha}K(t,x)=t^{-(n+|\alpha|)/2}D^{\alpha}K(1,\frac{x}{\sqrt{t}})$. But $y\mapsto D^{\alpha}K(1,y)$ is a continuous function on $\mathbb{R}^{n}$ satisfying $|D^{\alpha}K(1,y)|\lesssim{|y|}^{-(n+|\alpha|)}$ for $y\in\mathbb{R}^{n}\setminus\{0\}$, so $|D^{\alpha}K(1,y)|\lesssim{(1+|y|)}^{-(n+|\alpha|)}$ for $y\in\mathbb{R}^{n}$. Therefore $D^{\alpha}K(t,x)=t^{-(n+|\alpha|)/2}D^{\alpha}K(1,\frac{x}{\sqrt{t}})$ satisfies (\ref{heat-type-pointwise}).
\end{proof}
We deduce the following Lorentz estimates
\begin{lemma}
    For $p\in[1,\infty)$, the heat kernel satisfies the estimate
    \begin{equation}\label{heat-lorentz}
        {\|\Phi(t)\|}_{L^{p,1}(\mathbb{R}^{n})}^{*} \lesssim_{n} t^{-\frac{n}{2p'}}.
    \end{equation}
    More generally, for $p\in[1,\infty)$ and $\alpha\in\mathbb{Z}_{\geq0}^{n}$ with $|\alpha|>0$, the heat-type kernel satisfies
    \begin{equation}\label{heat-type-lorentz}
        {\|D^{\alpha}K(t)\|}_{L^{p,1}(\mathbb{R}^{n})}^{*} \lesssim_{\alpha,h,n} t^{-\frac{1}{2}\left(|\alpha|+\frac{n}{p'}\right)}.
    \end{equation}
\end{lemma}
\begin{proof}
    By (\ref{explicit-heat-kernel}) and (\ref{heat-type-pointwise}), the kernels $\Phi$ and $D^{\alpha}K$ are dominated by decreasing radial functions:
    \begin{equation}
        |\Phi(t,x)| \leq C_{1}t^{-\frac{n}{2}}e^{-\frac{{|x|}^{2}}{4t}}, \quad |D^{\alpha}K(t,x)| \leq C_{2}t^{-\frac{1}{2}(n+|\alpha|)}{\left(1+\frac{|x|}{\sqrt{t}}\right)}^{-(n+|\alpha|)},
    \end{equation}
    where the constants $C_{1},C_{2}$ are independent of $p$. We easily deduce the estimates
    \begin{equation}
        {[\Phi(t)]}^{*}(\tau) \leq C_{1}t^{-\frac{n}{2}}e^{-\frac{{(\tau/\omega_{n})}^{2/n}}{4t}}, \quad {[D^{\alpha}K(t)]}^{*}(\tau) \leq C_{2}t^{-\frac{1}{2}(n+|\alpha|)}{\left(1+\frac{{(\tau/\omega_{n})}^{1/n}}{\sqrt{t}}\right)}^{-(n+|\alpha|)},
    \end{equation}
    where $\omega_{n}$ is the Lebesgue measure of the unit ball in $\mathbb{R}^{n}$. We can therefore estimate
    \begin{equation}
    \begin{aligned}
        {\|\Phi(t)\|}_{L^{p,1}(\mathbb{R}^{n})}^{*} &\leq \frac{1}{p}\int_{0}^{\infty}\tau^{1/p}C_{1}t^{-\frac{n}{2}}e^{-\frac{{(\tau/\omega_{n})}^{2/n}}{4t}}\frac{\mathrm{d}\tau}{\tau} \\
        &= \frac{1}{p}\int_{0}^{\infty}{\left({(4t)}^{\frac{n}{2}}\omega_{n}\sigma\right)}^{1/p}C_{1}t^{-\frac{n}{2}}e^{-\sigma^{2/n}}\frac{\mathrm{d}\sigma}{\sigma} \\
        &\leq \left(1 \vee 2^{n}\omega_{n}\right)C_{1}t^{-\frac{n}{2p'}}\frac{1}{p}\int_{0}^{\infty}\sigma^{1/p}e^{-\sigma^{2/n}}\frac{\mathrm{d}\sigma}{\sigma} \\
        &\leq \left(1 \vee 2^{n}\omega_{n}\right)C_{1}t^{-\frac{n}{2p'}}\left(\frac{1}{p}\int_{0}^{1}\sigma^{1/p}\frac{\mathrm{d}\sigma}{\sigma} + \frac{1}{p}\int_{1}^{\infty}e^{-\sigma^{2/n}}\,\mathrm{d}\sigma\right) \\
        &\leq \left(1 \vee 2^{n}\omega_{n}\right)C_{1}t^{-\frac{n}{2p'}}\left(1 + \int_{1}^{\infty}e^{-\sigma^{2/n}}\,\mathrm{d}\sigma\right)
    \end{aligned}
    \end{equation}
    and
    \begin{equation}
    \begin{aligned}
        {\|D^{\alpha}K(t)\|}_{L^{p,1}(\mathbb{R}^{n})}^{*} &\leq \frac{1}{p}\int_{0}^{\infty}\tau^{1/p}C_{2}t^{-\frac{1}{2}(n+|\alpha|)}{\left(1+\frac{{(\tau/\omega_{n})}^{1/n}}{\sqrt{t}}\right)}^{-(n+|\alpha|)}\frac{\mathrm{d}\tau}{\tau} \\
        &= \frac{1}{p}\int_{0}^{\infty}{\left(t^{\frac{n}{2}}\omega_{n}\sigma\right)}^{1/p}C_{2}t^{-\frac{1}{2}(n+|\alpha|)}{\left(1+\sigma^{1/n}\right)}^{-(n+|\alpha|)}\frac{\mathrm{d}\sigma}{\sigma} \\
        &\leq \left(1\vee\omega_{n}\right)C_{2}t^{-\frac{1}{2}\left(|\alpha|+\frac{n}{p'}\right)}\frac{1}{p}\int_{0}^{\infty}\sigma^{1/p}{\left(1+\sigma^{1/n}\right)}^{-(n+|\alpha|)}\frac{\mathrm{d}\sigma}{\sigma} \\
        &\leq \left(1\vee\omega_{n}\right)C_{2}t^{-\frac{1}{2}\left(|\alpha|+\frac{n}{p'}\right)}\left(\frac{1}{p}\int_{0}^{1}\sigma^{1/p}\frac{\mathrm{d}\sigma}{\sigma} + \frac{1}{p}\int_{1}^{\infty}{\left(1+\sigma^{1/n}\right)}^{-(n+|\alpha|)}\,\mathrm{d}\sigma\right) \\
        &\leq \left(1\vee\omega_{n}\right)C_{2}t^{-\frac{1}{2}\left(|\alpha|+\frac{n}{p'}\right)}\left(1 + \int_{1}^{\infty}{\left(1+\sigma^{1/n}\right)}^{-(n+|\alpha|)}\,\mathrm{d}\sigma\right).
    \end{aligned}
    \end{equation}
\end{proof}
Having established that $D^{\beta}\Phi(s)\in L^{1}(\mathbb{R}^{n})$ for $\beta\in\mathbb{Z}_{\geq0}^{n}$ and $s\in(0,\infty)$, we use the Fourier inversion theorem to deduce that $\mathcal{F}[D^{\beta}\Phi(s)](\xi)=\xi^{\beta}e^{-s{|\xi|}^{2}}$. Combining this information with (\ref{heat-type-derivative}) and using Fubini's theorem, we obtain the following semigroup property.
\begin{lemma}
    For $\alpha,\beta\in\mathbb{Z}_{\geq0}^{n}$, we have
    \begin{equation}\label{heat-type-semigroup}
        D^{\beta}\Phi(s)*D^{\alpha}K(t)=D^{\alpha+\beta}K(s+t) \quad \text{for }s,t\in(0,\infty).
    \end{equation}
\end{lemma}
\begin{proof}
    We compute
    \begin{equation}
    \begin{aligned}
        \left[D^{\beta}\Phi(s)*D^{\alpha}K(t)\right](x) &= \int_{\mathbb{R}^{n}}D^{\beta}\Phi(s,y)D^{\alpha}K(t,x-y)\,\mathrm{d}y \\
        &= \int_{\mathbb{R}^{n}}D^{\beta}\Phi(s,y)\left(\frac{1}{{(2\pi)}^{n}}\int_{\mathbb{R}^{n}}e^{\mathrm{i}(x-y)\cdot\xi-t{|\xi|}^{2}}\xi^{\alpha}h(\xi)\,\mathrm{d}\xi\right)\,\mathrm{d}y \\
        &= \frac{1}{{(2\pi)}^{n}}\int_{\mathbb{R}^{n}}e^{\mathrm{i}x\cdot\xi-t{|\xi|}^{2}}\xi^{\alpha}h(\xi)\left(\int_{\mathbb{R}^{n}}e^{-\mathrm{i}y\cdot\xi}D^{\beta}\Phi(s,y)\,\mathrm{d}y\right)\,\mathrm{d}\xi \\
        &= \frac{1}{{(2\pi)}^{n}}\int_{\mathbb{R}^{n}}e^{\mathrm{i}x\cdot\xi-t{|\xi|}^{2}}\xi^{\alpha}h(\xi)\xi^{\beta}e^{-s{|\xi|}^{2}}\,\mathrm{d}\xi \\
        &= D^{\alpha+\beta}K(s+t,x),
    \end{aligned}
    \end{equation}
    where we use (\ref{heat-type-derivative}) in the second line, Fubini's theorem in the third, $\mathcal{F}[D^{\beta}\Phi(s)](\xi)=\xi^{\beta}e^{-s{|\xi|}^{2}}$ in the fourth, and (\ref{heat-type-derivative}) in the fifth. The use of Fubini's theorem is justified because the double integral is dominated by
    \begin{equation}
        \left(\int_{\mathbb{R}^{n}}|D^{\beta}\Phi(s,y)|\,\mathrm{d}y\right)\left(\int_{\mathbb{R}^{n}}e^{-t{|\xi|}^{2}}{|\xi|}^{|\alpha|}|h(\xi)|\,\mathrm{d}\xi\right).
    \end{equation}
\end{proof}
The semigroup property allows us to prove the following continuity result.
\begin{lemma}\label{heat-type-continuity}
    The heat kernel defines a continuous function $\Phi:(0,\infty)\rightarrow L^{1}(\mathbb{R}^{n})$. More generally, for $\alpha\in\mathbb{Z}_{\geq0}^{n}$ with $|\alpha|>0$, the heat-type kernel defines a continuous function $D^{\alpha}K:(0,\infty)\rightarrow L^{1}(\mathbb{R}^{n})$.
\end{lemma}
\begin{proof}
    Differentiating (\ref{explicit-heat-kernel}) with respect to $t$, we have
    \begin{equation}
        \Phi'(t,x) = \left(\frac{{|x|}^{2}}{4t^{2}}-\frac{n}{2t}\right)\Phi(t,x),
    \end{equation}
    so that, for each $x\in\mathbb{R}^{n}$ and $0<t_{0}<t_{1}<\infty$,
    \begin{equation}
        \sup_{t\in[t_{0},t_{1}]}\Phi(t,x) = \left\{\begin{array}{ll}
            \Phi(t_{0},x) & \text{if }{|x|}^{2}\leq 2nt_{0}, \\
            \Phi({|x|}^{2}/2n,x) & \text{if }2nt_{0}<{|x|}^{2}<2nt_{1}, \\
            \Phi(t_{1},x) & \text{if }2nt_{1}\leq{|x|}^{2},
        \end{array}\right.
    \end{equation}
    where $\Phi({|x|}^{2}/2n,x)\lesssim_{n}{|x|}^{-n}$. For each $0<t_{0}<t_{1}<\infty$, we deduce that the function $x\mapsto\sup_{t\in[t_{0},t_{1}]}\Phi(t,x)$ defines an element of $L^{1}(\mathbb{R}^{n})$. By (\ref{explicit-heat-kernel}), we see also that the function $t\mapsto\Phi(t,x)$ is continuous on $(0,\infty)$ for each $x\in\mathbb{R}^{n}$. By dominated convergence, we deduce that the function $\Phi:(0,\infty)\rightarrow L^{1}(\mathbb{R}^{n})$ is continuous.

    For $\alpha\in\mathbb{Z}_{\geq0}^{n}$ and $0<t_{0}<t<\infty$, the semigroup property yields $D^{\alpha}K(t)=\Phi(t-t_{0})*D^{\alpha}K(t_{0})$. Let $t_{0}$ be fixed, and allow $t\in(t_{0},\infty)$ to vary. If $|\alpha|>0$ then $D^{\alpha}K(t_{0})\in L^{1}(\mathbb{R}^{n})$. Using continuity of convolution $L^{1}(\mathbb{R}^{n})\times L^{1}(\mathbb{R}^{n})\rightarrow L^{1}(\mathbb{R}^{n})$, together with continuity of $\Phi:(0,\infty)\rightarrow L^{1}(\mathbb{R}^{n})$, we deduce that the function $D^{\alpha}K:(t_{0},\infty)\rightarrow L^{1}(\mathbb{R}^{n})$ is continuous. Since $t_{0}$ was arbitrary, we deduce that the function $D^{\alpha}K:(0,\infty)\rightarrow L^{1}(\mathbb{R}^{n})$ is continuous.
\end{proof}
\section{Integral operators}
Having established basic properties of Lorentz spaces and the kernels $\Phi$ and $\mathcal{T}$, we are able to study the integral operators
\begin{equation}\label{study-heat-map}
    S_{i}[f](t,x) := \int_{\mathbb{R}^{n}}\Phi(t,x-y)f_{i}(y)\,\mathrm{d}y,
\end{equation}
\begin{equation}\label{study-integral-operator}
    A_{i}[w](t,x) := \int_{0}^{t}\int_{\mathbb{R}^{n}}\nabla_{k}\mathcal{T}_{ij}(t-s,x-y)w_{jk}(s,y)\,\mathrm{d}y\,\mathrm{d}s,
\end{equation}
\begin{equation}\label{study-bilinear-map}
    B_{i}[u,v](t,x) := A_{i}[u\otimes v](t,x).
\end{equation}
For $T\in(0,\infty]$, $\sigma\in\mathbb{R}$, $p\in(1,\overline{\infty}]$, $\alpha\in(1,\infty]$, $q,\beta\in[1,\infty]$ satisfying $(p\in\{\infty,\overline{\infty}\}\Rightarrow q=\infty)$ and $(\alpha=\infty\Rightarrow \beta=\infty)$, and a measurable function $u$ on $(0,T)\times\mathbb{R}^{n}$, we define the norms

\pagebreak

\begin{equation}
    {\|u\|}_{\mathcal{L}_{p,q}^{\alpha,\beta}(T)} := {\left\|t\mapsto\mathbf{1}_{(0,T)}(t){\|u(t)\|}_{L^{p,q}(\mathbb{R}^{n})}\right\|}_{L^{\alpha,\beta}(\mathbb{R})},
\end{equation}
\begin{equation}
    {\|u\|}_{\mathcal{K}_{p,q}^{\sigma}(T)} := \esssup_{t\in(0,T)}t^{-\sigma/2}{\|u(t)\|}_{L^{p,q}(\mathbb{R}^{n})},
\end{equation}
\begin{equation}
    {\|u\|}_{\mathcal{J}_{p,q}^{\sigma}(T)} := \sup_{t\in(0,T)}t^{-\sigma/2}{\|u(t)\|}_{L^{p,q}(\mathbb{R}^{n})}.
\end{equation}
The quasinorms ${\|u\|}_{\mathcal{L}_{p,q}^{\alpha,\beta}(T)}^{*}$, ${\|u\|}_{\mathcal{K}_{p,q}^{\sigma}(T)}^{*}$ and ${\|u\|}_{\mathcal{J}_{p,q}^{\sigma}(T)}^{*}$ are obtained by replacing Lorentz norms with Lorentz quasinorms. We will make use of the following notation: if $u$ is a function defined on an interval $(0,T)$, and $t_{0}\in(0,T)$, then we define the time-shifted function $\tau_{t_{0}}u$ by setting $\tau_{t_{0}}u(t)=u(t+t_{0})$ for each $t\in(0,T-t_{0})$. The following lemma describes the properties of the heat map \eqref{study-heat-map}.
\begin{lemma}
    If $p\in(1,\infty]$ and $q\in[1,\infty]$ satisfy $(p=\infty\Rightarrow q=\infty)$, and $f\in L^{p,q}(\mathbb{R}^{n})$, then $S[f](t,x)$ is defined for all $(t,x)\in(0,\infty)\times\mathbb{R}^{n}$, and satisfies
    \begin{equation}\label{heat-estimate-1}
        {\|S[f]\|}_{\mathcal{J}_{p,q}^{0}(\infty)} \leq {\|f\|}_{L^{p,q}(\mathbb{R}^{n})},
    \end{equation}
    \begin{equation}\label{heat-estimate-2}
        {\|S[f]\|}_{\mathcal{J}_{\overline{\infty},\infty}^{-n/p}(\infty)} \lesssim_{n} p'{\|f\|}_{L^{p,\infty}(\mathbb{R}^{n})}^{*},
    \end{equation}
    \begin{equation}\label{heat-semigroup-property}
        \tau_{t_{0}}S[f] =_{\mathrm{e.}}^{\mathrm{e.}} S[S[f](t_{0})] \quad \text{for all }t_{0}\in(0,\infty).
    \end{equation}
\end{lemma}
\begin{proof}
    We estimate
    \begin{equation}
        {\|S[f](t)\|}_{L^{p,q}(\mathbb{R}^{n})}\leq{\|\Phi(t)\|}_{L^{1}(\mathbb{R}^{n})}{\|f\|}_{L^{p,q}(\mathbb{R}^{n})}={\|f\|}_{L^{p,q}(\mathbb{R}^{n})}
    \end{equation}
    and
    \begin{equation}
        {\|S[f](t)\|}_{L^{\overline{\infty}}(\mathbb{R}^{n})} \leq p'{\|\Phi(t)\|}_{L^{p',1}(\mathbb{R}^{n})}^{*}{\|f\|}_{L^{p,\infty}(\mathbb{R}^{n})}^{*} \lesssim_{n} p'{\|f\|}_{L^{p,\infty}(\mathbb{R}^{n})}^{*}.
     \end{equation}
    For $t_{0}\in(0,\infty)$, $t\in(0,\infty)$ and $x\in\mathbb{R}^{n}$ we have
    \begin{equation}
    \begin{aligned}
        \tau_{t_{0}}S[f](t,x) &= \int_{\mathbb{R}^{n}}\Phi(t+t_{0},y)f(x-y)\,\mathrm{d}y \\
        &= \int_{\mathbb{R}^{n}}\left(\int_{\mathbb{R}^{n}}\Phi(t,z)\Phi(t_{0},y-z)\,\mathrm{d}z\right)f(x-y)\,\mathrm{d}y \\
        &= \int_{\mathbb{R}^{n}}\Phi(t,z)\left(\int_{\mathbb{R}^{n}}\Phi(t_{0},y-z)f(x-y)\,\mathrm{d}y\right)\,\mathrm{d}z \\
        &= \int_{\mathbb{R}^{n}}\Phi(t,z)S[f](t_{0},x-z)\,\mathrm{d}z \\
        &= S[S[f](t_{0})](t,x),
    \end{aligned}
    \end{equation}
    where the use of Fubini's theorem in the third line is justified by the estimate
    \begin{equation}
    \begin{aligned}
        \int_{\mathbb{R}^{n}}\int_{\mathbb{R}^{n}}\Phi(t,z)\Phi(t_{0},y-z)|f(x-y)|\,\mathrm{d}z\,\mathrm{d}y &\leq {\|\Phi(t)\|}_{L^{1}(\mathbb{R}^{n})}{\|\Phi(t_{0})*|f|\|}_{L^{\infty}(\mathbb{R}^{n})} \\
        &\leq p'{\|\Phi(t)\|}_{L^{1}(\mathbb{R}^{n})}{\|\Phi(t_{0})\|}_{L^{p',1}(\mathbb{R}^{n})}^{*}{\|f\|}_{L^{p,\infty}(\mathbb{R}^{n})}^{*}.
    \end{aligned}
    \end{equation}
\end{proof}
Our product estimates for Lorentz spaces yields the inequalities
\begin{equation}\label{product-estimate-1}
    {\|u\otimes v\|}_{\mathcal{L}_{\frac{r}{2},\infty}^{\frac{\alpha}{2},1\vee\frac{\beta}{2}}(T)} \lesssim_{r,\alpha,\beta} {\|u\|}_{\mathcal{L}_{r,\infty}^{\alpha,\beta}(T)}{\|v\|}_{\mathcal{L}_{r,\infty}^{\alpha,\beta}(T)},
\end{equation}
\begin{equation}\label{product-estimate-2}
    {\|u\otimes v\|}_{\mathcal{K}_{\frac{r}{2},\infty}^{2\sigma}(T)}^{*} \leq {\|u\|}_{\mathcal{K}_{r,\infty}^{\sigma}(T)}{\|v\|}_{\mathcal{K}_{r,\infty}^{\sigma}(T)},
\end{equation}
\begin{equation}\label{product-estimate-3}
    {\|u\otimes v\|}_{\mathcal{K}_{p,q}^{\sigma}(T)} \leq {\|u\|}_{\mathcal{K}_{p,q}^{0}(T)}{\|v\|}_{\mathcal{K}_{\infty,\infty}^{\sigma}(T)},
\end{equation}
for $T\in(0,\infty]$, $\sigma\in\mathbb{R}$, $p\in(1,\infty]$, $r,\alpha\in(2,\infty]$, $q,\beta\in[1,\infty]$, $(p=\infty\Rightarrow q=\infty)$ and $(\alpha=\infty\Rightarrow\beta=\infty)$. The following lemma describes properties of the integral operator \eqref{study-integral-operator}, which combine with the product estimates \eqref{product-estimate-1}-\eqref{product-estimate-3} to give properties of the bilinear operator \eqref{study-bilinear-map}.
\begin{lemma}\label{semigrouplem}
    \begin{enumerate}[label=(\roman*)]
        \item
        Assume that $r\in(n,\infty]$, $\alpha\in(2,\infty)$, $\beta\in[1,\infty]$ and $\frac{2}{\alpha}+\frac{n}{r}\leq1$. Let $T\in(0,\infty)$, and let $w\in\mathcal{L}_{\frac{r}{2},\infty}^{\frac{\alpha}{2},1\vee\frac{\beta}{2}}(T)$. Then $A[w](t,x)$ is defined for almost every $(t,x)\in(0,T)\times\mathbb{R}^{n}$; in the case $r=\infty$, for almost every $t\in(0,T)$ we have that $A[w](t,x)$ is defined for all $x\in\mathbb{R}^{n}$. Moreover, we have the estimate
        \begin{equation}\label{uniqueness-bilinear-estimate}
            {\|A[w]\|}_{\mathcal{L}_{\overline{r},\infty}^{\alpha,\beta}(T)} \lesssim_{n,r,\alpha,\beta} T^{\frac{1}{2}\left(1-\frac{2}{\alpha}-\frac{n}{r}\right)}{\|w\|}_{\mathcal{L}_{\frac{r}{2},\infty}^{\frac{\alpha}{2},1\vee\frac{\beta}{2}}(T)},
        \end{equation}
        and the identity
        \begin{equation}\label{uniqueness-semigroup}
            \tau_{t_{0}}A[w] =_{\overline{r}}^{\mathrm{a.e.}} S[A[w](t_{0})] + A[\tau_{t_{0}}w] \quad \text{for almost every }t_{0}\in(0,T).
        \end{equation}
        \item
        Assume that $r\in(n,\infty]$, $\sigma\in(-1,\infty)$ and $-\sigma+\frac{n}{r}\leq1$. Let $T\in(0,\infty)$, and let $w\in\mathcal{K}_{\frac{r}{2},\infty}^{2\sigma}(T)$. Then for all $t\in(0,T)$ we have that $A[w](t,x)$ is defined for almost every $x\in\mathbb{R}^{n}$; in the case $r=\infty$, we have that $A[w](t,x)$ is defined for all $(t,x)\in(0,T)\times\mathbb{R}^{n}$. Moreover, we have the estimate
        \begin{equation}\label{existence-bilinear-estimate}
            {\|A[w]\|}_{\mathcal{J}_{\overline{r},\infty}^{\sigma}(T)} \lesssim_{n} {\left(\frac{r}{2}\right)}'T^{\frac{1}{2}\left(1+\sigma-\frac{n}{r}\right)}{\|w\|}_{\mathcal{K}_{\frac{r}{2},\infty}^{2\sigma}(T)}^{*}\int_{0}^{1}{(1-s)}^{-\frac{1}{2}\left(1+\frac{n}{r}\right)}s^{\sigma}\,\mathrm{d}s,
        \end{equation}
        and the identity
        \begin{equation}\label{existence-semigroup}
            \tau_{t_{0}}A[w] =_{\overline{r}}^{\mathrm{e.}} S[A[w](t_{0})] + A[\tau_{t_{0}}w] \quad \text{for all }t_{0}\in(0,T).
        \end{equation}
        If $w\in\mathcal{K}_{p,q}^{\sigma}(T)$ for some $p\in(1,\infty]$ and $q\in[1,\infty]$ satisfying $(p=\infty\Rightarrow q=\infty)$, then we have the estimate
        \begin{equation}\label{regularity-bilinear-estimate}
            {\|A[w]\|}_{\mathcal{J}_{p,q}^{0}(T)} \lesssim_{n} T^{\frac{1}{2}\left(1+\sigma\right)}{\|w\|}_{\mathcal{K}_{p,q}^{\sigma}(T)}\int_{0}^{1}{(1-s)}^{-\frac{1}{2}}s^{\frac{\sigma}{2}}\,\mathrm{d}s.
        \end{equation}
    \end{enumerate}
\end{lemma}
\begin{proof}
    We have
    \begin{equation}
    \begin{aligned}
        {\||\nabla\mathcal{T}(t-s)|*|w(s)|\|}_{L^{\frac{r}{2},\infty}(\mathbb{R}^{n})} &\leq {\|\nabla\mathcal{T}(t-s)\|}_{L^{1}(\mathbb{R}^{n})}{\|w(s)\|}_{L^{\frac{r}{2},\infty}(\mathbb{R}^{n})} \\
        &\lesssim_{n} {\left(\frac{r}{2}\right)}'{(t-s)}^{-\frac{1}{2}}{\|w(s)\|}_{L^{\frac{r}{2},\infty}(\mathbb{R}^{n})}^{*}
    \end{aligned}
    \end{equation}
    and
    \begin{equation}
    \begin{aligned}
        {\||\nabla\mathcal{T}(t-s)|*|w(s)|\|}_{L^{\overline{\infty}}(\mathbb{R}^{n})} &\leq {\left(\frac{r}{2}\right)}'{\|\nabla\mathcal{T}(t-s)\|}_{L^{{\left(\frac{r}{2}\right)}',1}(\mathbb{R}^{n})}^{*}{\|w(s)\|}_{L^{\frac{r}{2},\infty}(\mathbb{R}^{n})}^{*} \\
        &\lesssim_{n} {\left(\frac{r}{2}\right)}'{(t-s)}^{-\frac{1}{2}\left(1+\frac{2n}{r}\right)}{\|w(s)\|}_{L^{\frac{r}{2},\infty}(\mathbb{R}^{n})}^{*}.
    \end{aligned}
    \end{equation}
    Using the interpolation inequality ${\|g\|}_{L^{r,\infty}(\mathbb{R}^{n})}\leq{\|g\|}_{L^{\frac{r}{2},\infty}(\mathbb{R}^{n})}^{1/2}{\|g\|}_{L^{\infty,\infty}(\mathbb{R}^{n})}^{1/2}$, we deduce that
    \begin{equation}
        {\||\nabla\mathcal{T}(t-s)|*|w(s)|\|}_{L^{\overline{r},\infty}(\mathbb{R}^{n})} \lesssim_{n} {\left(\frac{r}{2}\right)}'{(t-s)}^{-\frac{1}{2}\left(1+\frac{n}{r}\right)}{\|w(s)\|}_{L^{\frac{r}{2},\infty}(\mathbb{R}^{n})}^{*},
    \end{equation}
    so by Minkowski's inequality we have
    \begin{equation}\label{bilinear-estimate-proof}
        {\|A[w](t)\|}_{L^{\overline{r},\infty}(\mathbb{R}^{n})} \lesssim_{n} {\left(\frac{r}{2}\right)}'\int_{0}^{t}{(t-s)}^{-\frac{1}{2}\left(1+\frac{n}{r}\right)}{\|w(s)\|}_{L^{\frac{r}{2},\infty}(\mathbb{R}^{n})}^{*}\,\mathrm{d}s.
    \end{equation}
    To prove \eqref{uniqueness-bilinear-estimate}, we observe that the right hand side of \eqref{bilinear-estimate-proof} is the convolution of $f(t):=\mathbf{1}_{(0,T)}(t)t^{-\frac{1}{2}\left(1+\frac{n}{r}\right)}$ and $g(t):=\mathbf{1}_{(0,T)}(t){\|w(t)\|}_{L^{\frac{r}{2},\infty}(\mathbb{R}^{n})}^{*}$, where by \eqref{local-embedding} we have
    \begin{equation}
        {\|f\|}_{L^{\alpha',\infty}(\mathbb{R})}^{*} \lesssim_{n,r,\alpha} T^{\frac{1}{\alpha'}-\frac{r+n}{2r}}{\|f\|}_{L^{\frac{2r}{r+n},\infty}(\mathbb{R})}^{*} = T^{\frac{1}{2}\left(1-\frac{2}{\alpha}-\frac{n}{r}\right)},
    \end{equation}
    so \eqref{uniqueness-bilinear-estimate} follows from the convolution inequality ${\|f*g\|}_{L^{\alpha,\beta}(\mathbb{R})}^{*}\lesssim_{\alpha,\beta}{\|f\|}_{L^{\alpha',\infty}(\mathbb{R})}^{*}{\|g\|}_{L^{\frac{\alpha}{2},1\vee\frac{\beta}{2}}(\mathbb{R})}^{*}$. On the other hand, estimate \eqref{existence-bilinear-estimate} follows from the calculation
    \begin{equation}
    \begin{aligned}
        t^{-\frac{\sigma}{2}}\int_{0}^{t}{(t-s)}^{-\frac{1}{2}\left(1+\frac{n}{r}\right)}{\|w(s)\|}_{L^{\frac{r}{2},\infty}(\mathbb{R}^{n})}^{*}\,\mathrm{d}s &\leq t^{-\frac{\sigma}{2}}{\|w\|}_{\mathcal{K}_{\frac{r}{2},\infty}^{2\sigma}(T)}^{*} \int_{0}^{t}{(t-s)}^{-\frac{1}{2}\left(1+\frac{n}{r}\right)}s^{\sigma}\,\mathrm{d}s \\
        &= t^{\frac{1}{2}\left(1+\sigma-\frac{n}{r}\right)}{\|w\|}_{\mathcal{K}_{\frac{r}{2},\infty}^{2\sigma}(T)}^{*}\int_{0}^{1}{(1-s)}^{-\frac{1}{2}\left(1+\frac{n}{r}\right)}s^{\sigma}\,\mathrm{d}s \\
        &\leq T^{\frac{1}{2}\left(1+\sigma-\frac{n}{r}\right)}{\|w\|}_{\mathcal{K}_{\frac{r}{2},\infty}^{2\sigma}(T)}^{*}\int_{0}^{1}{(1-s)}^{-\frac{1}{2}\left(1+\frac{n}{r}\right)}s^{\sigma}\,\mathrm{d}s.
    \end{aligned}
    \end{equation}
    Again using Minkowski's inequality, estimate \eqref{regularity-bilinear-estimate} follows from the calculation
    \begin{equation}
    \begin{aligned}
        {\|A[w](t)\|}_{L^{p,q}(\mathbb{R}^{n})} &\leq \int_{0}^{t}{\||\nabla\mathcal{T}(t-s)|*|w(s)|\|}_{L^{p,q}(\mathbb{R}^{n})}\,\mathrm{d}s \\
        &\leq \int_{0}^{t}{\|\nabla\mathcal{T}(t-s)\|}_{L^{1}(\mathbb{R}^{n})}{\|w(s)\|}_{L^{p,q}(\mathbb{R}^{n})}\,\mathrm{d}s \\
        &\lesssim_{n} \int_{0}^{t}{(t-s)}^{-\frac{1}{2}}{\|w(s)\|}_{L^{p,q}(\mathbb{R}^{n})}\,\mathrm{d}s \\
        &\leq {\|w\|}_{\mathcal{K}_{p,q}^{\sigma}(T)}\int_{0}^{t}{(t-s)}^{-\frac{1}{2}}s^{\frac{\sigma}{2}}\,\mathrm{d}s \\
        &= t^{\frac{1}{2}\left(1+\sigma\right)}{\|w\|}_{\mathcal{K}_{p,q}^{\sigma}(T)}\int_{0}^{1}{(1-s)}^{-\frac{1}{2}}s^{\frac{\sigma}{2}}\,\mathrm{d}s \\
        &\leq T^{\frac{1}{2}\left(1+\sigma\right)}{\|w\|}_{\mathcal{K}_{p,q}^{\sigma}(T)}\int_{0}^{1}{(1-s)}^{-\frac{1}{2}}s^{\frac{\sigma}{2}}\,\mathrm{d}s.
    \end{aligned}
    \end{equation}
    Finally, identities \eqref{uniqueness-semigroup} and \eqref{existence-semigroup} are based upon the calculation
    \begin{equation}
    \begin{aligned}
        \tau_{t_{0}}A_{i}[w](t,x) &= \int_{0}^{t_{0}+t}\int_{\mathbb{R}^{n}}\nabla_{k}\mathcal{T}_{ij}(t+t_{0}-s,y)w_{jk}(s,x-y)\,\mathrm{d}y\,\mathrm{d}s \\
        &= \int_{0}^{t_{0}}\int_{\mathbb{R}^{n}}\left(\int_{\mathbb{R}^{n}}\Phi(t,z)\nabla_{k}\mathcal{T}_{ij}(t_{0}-s,y-z)\,\mathrm{d}z\right)w_{jk}(s,x-y)\,\mathrm{d}y\,\mathrm{d}s \\
        &\qquad+ \int_{t_{0}}^{t_{0}+t}\int_{\mathbb{R}^{n}}\nabla_{k}\mathcal{T}_{ij}(t+t_{0}-s,y)w_{jk}(s,x-y)\,\mathrm{d}y\,\mathrm{d}s \\
        &= \int_{\mathbb{R}^{n}}\Phi(t,z)\left(\int_{0}^{t_{0}}\int_{\mathbb{R}^{n}}\nabla_{k}\mathcal{T}_{ij}(t_{0}-s,y-z)w_{jk}(s,x-y)\,\mathrm{d}y\,\mathrm{d}s\right)\,\mathrm{d}z \\
        &\qquad+ \int_{0}^{t}\int_{\mathbb{R}^{n}}\nabla_{k}\mathcal{T}_{ij}(t-\tilde{s},y)w_{jk}(\tilde{s}+t_{0},x-y)\,\mathrm{d}y\,\mathrm{d}\tilde{s} \\
        &= \int_{\mathbb{R}^{n}}\Phi(t,z)A_{i}[w](t_{0},x-z)\,\mathrm{d}z \\
        &\qquad + \int_{0}^{t}\int_{\mathbb{R}^{n}}\nabla_{k}\mathcal{T}_{ij}(t-\tilde{s},y)\tau_{t_{0}}w_{jk}(\tilde{s},x-y)\,\mathrm{d}y\,\mathrm{d}\tilde{s} \\
        &= S[A[w](t_{0})](t,x) + A[\tau_{t_{0}}w](t,x),
    \end{aligned}
    \end{equation}
    where the use of Fubini's theorem in the third line can be justified using our estimates for the operators $S$ and $A$.
\end{proof}
\section{Uniqueness}
By \eqref{product-estimate-1} and \eqref{uniqueness-bilinear-estimate}, if $r\in(n,\overline{\infty}]$, $\alpha\in(2,\infty)$, $\beta\in[1,\infty)$ and $\frac{2}{\alpha}+\frac{n}{r}\leq1$, then there exists a constant $C=C_{n,r,\alpha,\beta}$ such that
\begin{equation}\label{uni-bilinear}
    {\|B[u,v]\|}_{\mathcal{L}_{r,\infty}^{\alpha,\beta}(T)} \leq CT^{\frac{1}{2}\left(1-\frac{2}{\alpha}-\frac{n}{r}\right)}{\|u\|}_{\mathcal{L}_{r,\infty}^{\alpha,\beta}(T)}{\|v\|}_{\mathcal{L}_{r,\infty}^{\alpha,\beta}(T)} \quad \text{for all }T\in(0,\infty)\text{ and }u,v\in\mathcal{L}_{r,\infty}^{\alpha,\beta}(T).
\end{equation}
\begin{lemma}
    Assume that $r\in(n,\overline{\infty}]$, $\alpha\in(2,\infty)$, $\beta\in[1,\infty)$ and $\frac{2}{\alpha}+\frac{n}{r}\leq1$. Let $C=C_{n,r,\alpha,\beta}$ and $T\in(0,\infty)$. If $u^{0}\in\mathcal{L}_{r,\infty}^{\alpha,\beta}(T)$ satisfies $4CT^{\frac{1}{2}\left(1-\frac{2}{\alpha}-\frac{n}{r}\right)}{\|u^{0}\|}_{\mathcal{L}_{r,\infty}^{\alpha,\beta}(T)}<1$, then any solution $u\in\mathcal{L}_{r,\infty}^{\alpha,\beta}(T)$ to the equation $u=_{r}^{\mathrm{a.e.}}u^{0}-B[u,u]$ satisfies $CT^{\frac{1}{2}\left(1-\frac{2}{\alpha}-\frac{n}{r}\right)}{\|u\|}_{\mathcal{L}_{r,\infty}^{\alpha,\beta}(T)}\leq\Lambda$, where $\Lambda\in[0,\frac{1}{2})$ is the smaller root of the quadratic
    \begin{equation}\label{uni-Lambda}
        \Lambda = CT^{\frac{1}{2}\left(1-\frac{2}{\alpha}-\frac{n}{r}\right)}{\|u^{0}\|}_{\mathcal{L}_{r,\infty}^{\alpha,\beta}(T)} + \Lambda^{2}.
    \end{equation}
\end{lemma}
\begin{proof}
    Define
    \begin{equation}
        g_{0} := CT^{\frac{1}{2}\left(1-\frac{2}{\alpha}-\frac{n}{r}\right)}{\|u^{0}\|}_{\mathcal{L}_{r,\infty}^{\alpha,\beta}(T)}, \quad \text{and} \quad g(t) := Ct^{\frac{1}{2}\left(1-\frac{2}{\alpha}-\frac{n}{r}\right)}{\|u\|}_{\mathcal{L}_{r,\infty}^{\alpha,\beta}(t)} \text{ for }t\in(0,T],
    \end{equation}
    and let $\Gamma$ be the larger root of the quadratic $\Gamma=g_{0}+\Gamma^{2}$. Since $u=_{r}^{\mathrm{a.e.}}u^{0}-B[u,u]$, the bilinear estimate \eqref{uni-bilinear} yields $g(t)\leq g_{0}+{g(t)}^{2}$, so $g$ takes values in $[0,\Lambda]\cup[\Gamma,\infty)$. By dominated convergence in $L^{\alpha,\beta}(\mathbb{R})$, the function $g$ is continuous on $(0,T]$, satisfying $g(t)\rightarrow0$ as $t\rightarrow0$. Therefore $g(t)\leq\Lambda$ for all $t\in(0,T]$.
\end{proof}
\begin{lemma}
    Assume that $r\in(n,\overline{\infty}]$, $\alpha\in(2,\infty)$, $\beta\in[1,\infty)$ and $\frac{2}{\alpha}+\frac{n}{r}\leq1$. Let $C=C_{n,r,\alpha,\beta}$ and $T\in(0,\infty)$. If $u^{0}\in\mathcal{L}_{r,\infty}^{\alpha,\beta}(T)$ satisfies $4CT^{\frac{1}{2}\left(1-\frac{2}{\alpha}-\frac{n}{r}\right)}{\|u^{0}\|}_{\mathcal{L}_{r,\infty}^{\alpha,\beta}(T)}<1$, and $u,v\in\mathcal{L}_{r,\infty}^{\alpha,\beta}(T)$ satisfy $u^{0}=_{r}^{\mathrm{a.e.}}u+B[u,u]=_{r}^{\mathrm{a.e.}}v+B[v,v]$, then $u=_{r}^{\mathrm{a.e.}}v$.
\end{lemma}
\begin{proof}
    Let $u,v$ be as described. Then
    \begin{equation}
        u-v=_{r}^{\mathrm{a.e.}}-B[u,u]+B[v,v]=_{r}^{\mathrm{a.e.}}-B[u,u-v]-B[u-v,v].
    \end{equation}
    By the bilinear estimate \eqref{uni-bilinear} and the previous lemma, we deduce that
    \begin{equation}
    \begin{aligned}
        {\|u-v\|}_{\mathcal{L}_{r,\infty}^{\alpha,\beta}(T)} &\leq CT^{\frac{1}{2}\left(1-\frac{2}{\alpha}-\frac{n}{r}\right)}\left({\|u\|}_{\mathcal{L}_{r,\infty}^{\alpha,\beta}(T)}+{\|v\|}_{\mathcal{L}_{r,\infty}^{\alpha,\beta}(T)}\right){\|u-v\|}_{\mathcal{L}_{r,\infty}^{\alpha,\beta}(T)} \\
        &\leq 2\Lambda{\|u-v\|}_{\mathcal{L}_{r,\infty}^{\alpha,\beta}(T)},
    \end{aligned}
    \end{equation}
    where $\Lambda\in[0,\frac{1}{2})$ is the smaller root of the quadratic \eqref{uni-Lambda}. Therefore ${\|u-v\|}_{\mathcal{L}_{r,\infty}^{\alpha,\beta}(T)}=0$, so $u=_{r}^{\mathrm{a.e.}}v$.
\end{proof}
\begin{theorem}\label{uniqthm} (Uniqueness).
    Assume that $p\in(1,\infty]$, $r\in(n,\overline{\infty}]$, $\alpha\in(2,\infty)$, $\beta\in[1,\infty)$, and $\frac{2}{\alpha}+\frac{n}{r}\leq1$. If $T\in(0,\infty]$, $f\in L^{p,\infty}(\mathbb{R}^{n})$, and $u,v\in\mathcal{L}_{r,\infty}^{\alpha,\beta}(T_{-})$ satisfy $S[f] =_{r}^{\mathrm{a.e.}} u+B[u,u] =_{r}^{\mathrm{a.e.}} v+B[v,v]$, then $u =_{r}^{\mathrm{a.e.}} v$.
\end{theorem}
\begin{proof}
    Let $u,v$ be as described, and define
    \begin{equation}
        T_{0} := \sup\{T'\in(0,T)\text{ : }u(t)=_{r}v(t)\text{ for a.e.\ }t\in(0,T')\} \in[0,T].
    \end{equation}
    By the previous lemma, if $4C{(T')}^{\frac{1}{2}\left(1-\frac{2}{\alpha}-\frac{n}{r}\right)}{\|S[f]\|}_{\mathcal{L}_{r,\infty}^{\alpha,\beta}(T')}<1$ then we have uniqueness on the time interval $(0,T')$. By dominated convergence in $L^{\alpha,\beta}(\mathbb{R})$, this inequality holds for sufficiently small $T'$, so $T_{0}>0$.

    Suppose for the sake of contradiction that $T_{0}<T$. By dominated convergence in $L^{\alpha,\beta}(\mathbb{R})$, for $\epsilon>0$ sufficiently small we have
    \begin{equation}\label{uniqueness-proof}
        4C{(2\epsilon)}^{\frac{1}{2}\left(1-\frac{2}{\alpha}-\frac{n}{r}\right)}\left({\|\tau_{T_{0}-\epsilon}u\|}_{\mathcal{L}_{r,\infty}^{\alpha,\beta}(2\epsilon)}+C{(2\epsilon)}^{\frac{1}{2}\left(1-\frac{2}{\alpha}-\frac{n}{r}\right)}{\|\tau_{T_{0}-\epsilon}u\|}_{\mathcal{L}_{r,\infty}^{\alpha,\beta}(2\epsilon)}^{2}\right) < 1.
    \end{equation}
    By \eqref{heat-semigroup-property} and \eqref{uniqueness-semigroup}, we can choose a suitable value of $\epsilon$ satisfying \eqref{uniqueness-proof} such that $u(T_{0}-\epsilon)=_{r}v(T_{0}-\epsilon)$, and such that the functions $\tau_{T_{0}-\epsilon}u,\tau_{T_{0}-\epsilon}v\in\mathcal{L}_{r,\infty}^{\alpha,\beta}(2\epsilon)$ satisfy $S[u(T_{0}-\epsilon)]=_{r}^{\mathrm{a.e.}}\tau_{T_{0}-\epsilon}u+B[\tau_{T_{0}-\epsilon}u,\tau_{T_{0}-\epsilon}u]=_{r}^{\mathrm{a.e.}}\tau_{T_{0}-\epsilon}v+[\tau_{T_{0}-\epsilon}v,\tau_{T_{0}-\epsilon}v]$. By \eqref{uni-bilinear} and \eqref{uniqueness-proof}, we then have $4C{(2\epsilon)}^{\frac{1}{2}\left(1-\frac{2}{\alpha}-\frac{n}{r}\right)}{\|S[u(T_{0}-\epsilon)]\|}_{\mathcal{L}_{r,\infty}^{\alpha,\beta}(2\epsilon)}<1$. By the previous lemma, it follows that $u=_{r}^{\mathrm{a.e.}}v$ on the time interval $(T_{0}-\epsilon,T_{0}+\epsilon)$, so $u=_{r}^{\mathrm{a.e.}}v$ on $(0,T_{0}+\epsilon)$, which contradicts the definition of $T_{0}$.
\end{proof}
\section{Local existence and blowup rates}
For $r\in(n,\infty]$, define the constants $\alpha_{n}$, $\beta_{n/r}$, $\gamma_{n}$, $\delta_{n,r}$, $\eta_{n,r}$ such that
\begin{equation}
    {\|\Phi(t)\|}_{L^{r',1}(\mathbb{R}^{n})}^{*} \leq \alpha_{n}t^{-\frac{n}{2r}}, \quad \beta_{n/r} = \int_{0}^{1}{(1-s)}^{-\frac{1}{2}}s^{-\frac{n}{r}}\,\mathrm{d}s, \quad {\|\nabla\mathcal{T}(t)\|}_{L^{1}(\mathbb{R}^{n})} \leq \gamma_{n}t^{-\frac{1}{2}},
\end{equation}
\begin{equation}
    \delta_{n,r} = \beta_{n/r}\gamma_{n}, \quad \eta_{n,r} = r'\alpha_{n}\delta_{n,r}.
\end{equation}
We note that $\alpha_{n}$ may be chosen independently of $r$ by \eqref{heat-lorentz}. By \eqref{heat-estimate-1}, \eqref{heat-estimate-2}, \eqref{product-estimate-2}, \eqref{product-estimate-3}, \eqref{existence-bilinear-estimate} and \eqref{regularity-bilinear-estimate}, we have the estimates
\begin{equation}\label{exi-heat}
    {\|S[f]\|}_{\mathcal{J}_{\overline{\infty},\infty}^{-n/r}(\infty)} \leq r'\alpha_{n}{\|f\|}_{L^{r,\infty}(\mathbb{R}^{n})}^{*}, \quad {\|S[f]\|}_{\mathcal{J}_{p,q}^{0}(\infty)}\leq{\|f\|}_{L^{p,q}(\mathbb{R}^{n})},
\end{equation}
\begin{equation}\label{exi-bilinear-1}
    {\|B[u,v]\|}_{\mathcal{J}_{\overline{\infty},\infty}^{-n/r}(T)} \leq \delta_{n,r}T^{\frac{1}{2}\left(1-\frac{n}{r}\right)}{\|u\|}_{\mathcal{J}_{\overline{\infty},\infty}^{-n/r}(T)}{\|v\|}_{\mathcal{J}_{\overline{\infty},\infty}^{-n/r}(T)},
\end{equation}
\begin{equation}\label{exi-bilinear-2}
    {\|B[u,v]\|}_{\mathcal{J}_{p,q}^{0}(T)} \leq \delta_{n,r}T^{\frac{1}{2}\left(1-\frac{n}{r}\right)}\left({\|u\|}_{\mathcal{J}_{p,q}^{0}(T)}{\|v\|}_{\mathcal{J}_{\overline{\infty},\infty}^{-n/r}(T)} \wedge {\|u\|}_{\mathcal{J}_{\overline{\infty},\infty}^{-n/r}(T)}{\|v\|}_{\mathcal{J}_{p,q}^{0}(T)}\right)
\end{equation}
for $p\in(1,\infty]$, $q\in[1,\infty]$, $(p=\infty\Rightarrow q=\infty)$, $r\in(n,\infty]$ and $T\in(0,\infty)$.
\begin{theorem} (Local existence).
    Assume that $p\in(1,\infty]$, $q\in[1,\infty]$, $(p=\infty\Rightarrow q=\infty)$, $r\in(n,\infty]$ and $T\in(0,\infty)$. If $u^{0}\in\mathcal{J}_{p,q}^{0}(T)\cap\mathcal{J}_{\overline{\infty},\infty}^{-n/r}(T)$ and $4\delta_{n,r}T^{\frac{1}{2}\left(1-\frac{n}{r}\right)}{\|u^{0}\|}_{\mathcal{J}_{\overline{\infty},\infty}^{-n/r}(T)}<1$ (which occurs if $f\in L^{p,q}(\mathbb{R}^{n})\cap L^{r,\infty}(\mathbb{R}^{n})$, $u^{0}=S[f]$ and $4\eta_{n,r}T^{\frac{1}{2}\left(1-\frac{n}{r}\right)}{\|f\|}_{L^{r,\infty}(\mathbb{R}^{n})}<1$), then there exists $u\in\mathcal{J}_{p,q}^{0}(T)\cap\mathcal{J}_{\overline{\infty},\infty}^{-n/r}(T)$ satisfying $u=_{\mathrm{e.}}^{\mathrm{e.}}u^{0}-B[u,u]$.
\end{theorem}
\begin{proof}
    Define inductively, for each $m\in\mathbb{Z}_{\geq0}$, $u^{m+1}:=u^{0}-B[u^{m},u^{m}]\in\mathcal{J}_{p,q}^{0}(T)\cap\mathcal{J}_{\overline{\infty},\infty}^{-n/r}(T)$ (defined pointwise for all $(t,x)\in(0,T)\times\mathbb{R}^{n}$). By \eqref{exi-bilinear-1}, for each $m\in\mathbb{Z}_{\geq0}$ we have
    \begin{equation}
        \delta_{n,r}T^{\frac{1}{2}\left(1-\frac{n}{r}\right)}{\|u^{m+1}\|}_{\mathcal{J}_{\overline{\infty},\infty}^{-n/r}(T)} \leq \delta_{n,r}T^{\frac{1}{2}\left(1-\frac{n}{r}\right)}{\|u^{0}\|}_{\mathcal{J}_{\overline{\infty},\infty}^{-n/r}(T)} + {\left(\delta_{n,r}T^{\frac{1}{2}\left(1-\frac{n}{r}\right)}{\|u^{m}\|}_{\mathcal{J}_{\overline{\infty},\infty}^{-n/r}(T)}\right)}^{2},
    \end{equation}
    so by induction on $m$ we have that $\delta_{n,r}T^{\frac{1}{2}\left(1-\frac{n}{r}\right)}{\|u^{m}\|}_{\mathcal{J}_{\overline{\infty},\infty}^{-n/r}(T)}$ is bounded above by the smaller solution $\Lambda\in[0,\frac{1}{2})$ of the quadratic
    \begin{equation}
        \Lambda = \delta_{n,r}T^{\frac{1}{2}\left(1-\frac{n}{r}\right)}{\|u^{0}\|}_{\mathcal{J}_{\overline{\infty},\infty}^{-n/r}(T)} + \Lambda^{2}.
    \end{equation}
    For each $m\in\mathbb{Z}_{>0}$ we have
    \begin{equation}
        u^{m+1} - u^{m} = -B[u^{m},u^{m}] + B[u^{m-1},u^{m-1}] = - B[u^{m},u^{m}-u^{m-1}] - B[u^{m}-u^{m-1},u^{m-1}],
    \end{equation}
    so by \eqref{exi-bilinear-1} we have
    \begin{equation}
    \begin{aligned}
        {\|u^{m+1} - u^{m}\|}_{\mathcal{J}_{\overline{\infty},\infty}^{-n/r}(T)} &\leq \delta_{n,r}T^{\frac{1}{2}\left(1-\frac{n}{r}\right)}\left({\|u^{m}\|}_{\mathcal{J}_{\overline{\infty},\infty}^{-n/r}(T)}+{\|u^{m-1}\|}_{\mathcal{J}_{\overline{\infty},\infty}^{-n/r}(T)}\right){\|u^{m} - u^{m-1}\|}_{\mathcal{J}_{\overline{\infty},\infty}^{-n/r}(T)} \\
        &\leq 2\Lambda{\|u^{m} - u^{m-1}\|}_{\mathcal{J}_{\overline{\infty},\infty}^{-n/r}(T)},
    \end{aligned}
    \end{equation}
    and by \eqref{exi-bilinear-2} we have
    \begin{equation}
    \begin{aligned}
        {\|u^{m+1} - u^{m}\|}_{\mathcal{J}_{p,q}^{0}(T)} &\leq \delta_{n,r}T^{\frac{1}{2}\left(1-\frac{n}{r}\right)}\left({\|u^{m}\|}_{\mathcal{J}_{\overline{\infty},\infty}^{-n/r}(T)}+{\|u^{m-1}\|}_{\mathcal{J}_{\overline{\infty},\infty}^{-n/r}(T)}\right){\|u^{m} - u^{m-1}\|}_{\mathcal{J}_{p,q}^{0}(T)} \\
        &\leq 2\Lambda{\|u^{m} - u^{m-1}\|}_{\mathcal{J}_{p,q}^{0}(T)}.
    \end{aligned}
    \end{equation}
    Since $2\Lambda<1$, we deduce that $\sum_{m=1}^{\infty}{\|u^{m+1} - u^{m}\|}_{\mathcal{J}_{\overline{\infty},\infty}^{-n/r}(T)}$ and $\sum_{m=1}^{\infty}{\|u^{m+1} - u^{m}\|}_{\mathcal{J}_{p,q}^{0}(T)}$ converge, so $u^{m}$ converges in $\mathcal{J}_{p,q}^{0}(T)\cap\mathcal{J}_{\overline{\infty},\infty}^{-n/r}(T)$ to some $u\in\mathcal{J}_{p,q}^{0}(T)\cap\mathcal{J}_{\overline{\infty},\infty}^{-n/r}(T)$. By continuity of $B:\mathcal{J}_{\overline{\infty},\infty}^{-n/r}(T)\times\mathcal{J}_{\overline{\infty},\infty}^{-n/r}(T)\rightarrow\mathcal{J}_{\overline{\infty},\infty}^{-n/r}(T)$, the limit satisfies $u=_{\mathrm{e.}}^{\mathrm{e.}}u^{0}-B[u,u]$.
\end{proof}
\begin{theorem} (Blowup rates).
    Assume that $p\in(1,\infty]$, $q\in[1,\infty]$, $(p=\infty\Rightarrow q=\infty)$ and $r_{0}\in(n,\infty]$. Let $f\in L^{p,q}(\mathbb{R}^{n})\cap L^{r_{0},\infty}(\mathbb{R}^{n})$, and let $T\in(0,\infty]$ be the maximal time for which there exists $u\in\mathcal{J}_{p,q}^{0}(T_{-})\cap\mathcal{J}_{\overline{\infty},\infty}^{-n/r_{0}}(T_{-})$ satisfying $u =_{\mathrm{e.}}^{\mathrm{e.}} S[f] - B[u,u]$. If $T<\infty$, then the unique solution $u\in\mathcal{J}_{p,q}^{0}(T_{-})\cap\mathcal{J}_{\overline{\infty},\infty}^{-n/r_{0}}(T_{-})$ satisfies
    \begin{equation}
        {\|u(t_{0})\|}_{L^{r,\infty}(\mathbb{R}^{n})}^{*} \geq \frac{1}{4\eta_{n,r}{(T-t_{0})}^{\frac{1}{2}\left(1-\frac{n}{r}\right)}} \quad \text{for all }r\in(n,\infty]\text{ and }t_{0}\in(0,T),
    \end{equation}
    where we use the convention that ${\|u(t_{0})\|}_{L^{r,\infty}(\mathbb{R}^{n})}^{*}=\infty$ if $u(t_{0})\notin L^{r,\infty}(\mathbb{R}^{n})$.
\end{theorem}
\begin{proof}
    The previous theorem ensures that $T>0$. Suppose that $T<\infty$ and, for the sake of contradiction, that there exist $t_{0}\in(0,T)$ and $r\in(n,\infty]$ such that $4\eta_{n,r}{(T-t_{0})}^{\frac{1}{2}\left(1-\frac{n}{r}\right)}{\|u(t_{0})\|}_{L^{r,\infty}(\mathbb{R}^{n})}^{*}<1$. Then we can choose $T_{0}\in(T,\infty)$ sufficiently close to $T$ such that $4\eta_{n,r}{(T_{0}-t_{0})}^{\frac{1}{2}\left(1-\frac{n}{r}\right)}{\|u(t_{0})\|}_{L^{r,\infty}(\mathbb{R}^{n})}^{*}<1$. By the previous theorem, there exists $w\in\mathcal{J}_{p,q}^{0}(T_{0}-t_{0})\cap\mathcal{J}_{\overline{\infty},\infty}^{-n/r}(T_{0}-t_{0})$ satisfying $w=_{\mathrm{e.}}^{\mathrm{e.}}S[u(t_{0})]-B[w,w]$. By \eqref{heat-semigroup-property} and \eqref{existence-semigroup}, the function $\tau_{t_{0}}u\in\mathcal{J}_{\overline{\infty},\infty}^{0}({(T-t_{0})}_{-})\subseteq\mathcal{J}_{\overline{\infty},\infty}^{-n/r}({(T-t_{0})}_{-})$ satisfies $\tau_{t_{0}}u=_{\mathrm{e.}}^{\mathrm{e.}}S[u(t_{0})]-B[\tau_{t_{0}}u,\tau_{t_{0}}u]$. By uniqueness in $\mathcal{J}_{\overline{\infty},\infty}^{-n/r}({(T-t_{0})}_{-})$ (Corollary \ref{intro-uniqueness-corollary}), we deduce that $u(t+t_{0})=_{\mathrm{e.}}w(t)$ for all $t\in(0,T-t_{0})$. Therefore
    \begin{equation}
        v(t) := \left\{\begin{array}{ll} u(t) & \text{if }t\in(0,T), \\ w(t-t_{0}) & \text{if }t\in(t_{0},T) \end{array}\right.
    \end{equation}
    determines a well-defined function $v\in\mathcal{J}_{p,q}^{0}(T_{0})\cap\mathcal{J}_{\overline{\infty},\infty}^{-n/r}(T_{0})$. By \eqref{heat-semigroup-property} and \eqref{existence-semigroup} we have $v=_{\mathrm{e.}}^{\mathrm{e.}}S[f]-B[v,v]$, which contradicts maximality of $T$.
\end{proof}
\section{Continuity of solutions}
If $p\in(1,\infty]$, $q\in[1,\infty]$, $(p=\infty\Rightarrow q=\infty)$, $r\in(n,\infty]$, and $f\in L^{p,q}(\mathbb{R}^{n})\cap L^{r,\infty}(\mathbb{R}^{n})$, then there exists (at least locally in time) a solution $u\in\mathcal{J}_{p,q}^{0}(T_{-})\cap\mathcal{J}_{\overline{\infty},\infty}^{-n/r}(T_{-})$ to the equation $u=_{\mathrm{e.}}^{\mathrm{e.}}S[f]-B[u,u]$. The bilinear estimate ${\|B[u,u](t)\|}_{L^{p,q}(\mathbb{R}^{n})}\leq\delta_{n,r}t^{\frac{1}{2}\left(1-\frac{n}{r}\right)}{\|u\|}_{\mathcal{J}_{p,q}^{0}(t)}{\|u(t)\|}_{\mathcal{J}_{\overline{\infty},\infty}^{-n/r}(t)}$ ensures that ${\|B[u,u](t)\|}_{L^{p,q}(\mathbb{R}^{n})}\rightarrow0$ as $t\rightarrow0$. In the case $q<\infty$, approximation of identity ensures that ${\|S[f](t)-f\|}_{L^{p,q}(\mathbb{R}^{n})}\rightarrow0$ as $t\rightarrow0$, which implies that ${\|u(t)-f\|}_{L^{p,q}(\mathbb{R}^{n})}\rightarrow0$ as $t\rightarrow0$.

We now turn our attention to continuity away from the initial time. By \eqref{heat-semigroup-property} and \eqref{existence-semigroup}, the solution $u$ in the previous paragraph satisfies
\begin{equation}\label{continuity-semigroup}
    u_{i}(t,x) = \int_{\mathbb{R}^{n}}\Phi(t-t_{0},x-y)u_{i}(t_{0},y)\,\mathrm{d}y - \int_{t_{0}}^{t}\int_{\mathbb{R}^{n}}\nabla_{k}\mathcal{T}_{ij}(t-s,x-y)u_{j}(s,y)u_{k}(s,y)\,\mathrm{d}y\,\mathrm{d}s
\end{equation}
for all $0<t_{0}<t<T$ and $x\in\mathbb{R}^{n}$, with ${\|u(t)\|}_{L^{p,q}(\mathbb{R}^{n})}$ and ${\|u(t)\|}_{L^{\overline{\infty}}(\mathbb{R}^{n})}$ being locally bounded in $t\in(0,T)$.
\begin{theorem} (Continuity in time).
    Suppose that $p\in(1,\infty]$, $q\in[1,\infty]$ and $(p=\infty\Rightarrow q=\infty)$. Assume that $u$ satisfies \eqref{continuity-semigroup} for all $0<t_{0}<t<T$ and $x\in\mathbb{R}^{n}$, with ${\|u(t)\|}_{L^{p,q}(\mathbb{R}^{n})}$ and ${\|u(t)\|}_{L^{\overline{\infty}}(\mathbb{R}^{n})}$ being locally bounded for $t\in(0,T)$. Then $u$ defines a continuous function $u:(0,T)\rightarrow L^{p,q}(\mathbb{R}^{n})\cap L^{\overline{\infty}}(\mathbb{R}^{n})$.
\end{theorem}
\begin{proof}
    We use the shorthand notation
    \begin{equation}
        u_{i}(t) = \Phi(t-t_{0})*u_{i}(t_{0}) - \int_{0}^{t}\nabla_{k}\mathcal{T}_{ij}(t-s)*[u_{j}(s)u_{k}(s)]\,\mathrm{d}s.
    \end{equation}
    The heat kernel defines a continuous function $\Phi:(0,\infty)\rightarrow L^{1}(\mathbb{R}^{n})$. By continuity of convolution $L^{1}(\mathbb{R}^{n})\times\left(L^{p,q}(\mathbb{R}^{n})\cap L^{\overline{\infty}}(\mathbb{R}^{n})\right)\rightarrow L^{p,q}(\mathbb{R}^{n})\cap L^{\overline{\infty}}(\mathbb{R}^{n})$, we deduce that $t\rightarrow\Phi(t-t_{0})*u_{i}(t_{0})$ defines a continuous function from $(t_{0},\infty)$ to $L^{p,q}(\mathbb{R}^{n})\cap L^{\overline{\infty}}(\mathbb{R}^{n})$.

    It remains for us to prove continuity of the bilinear term $z_{i}(t):=\int_{0}^{t}\nabla_{k}\mathcal{T}_{ij}(t-s)*[u_{j}(s)u_{k}(s)]\,\mathrm{d}s$. Since ${\|u(t)\|}_{L^{p,q}(\mathbb{R}^{n})\cap L^{\overline{\infty}}(\mathbb{R}^{n})}:={\|u(t)\|}_{L^{p,q}(\mathbb{R}^{n})}\vee{\|u(t)\|}_{L^{\overline{\infty}}(\mathbb{R}^{n})}$ is locally bounded for $t\in(0,T)$, for each $T_{0}\in(t_{0},T)$ we can choose $M\in(0,\infty)$ such that
    \begin{equation}
        \sup_{t\in(t_{0},T_{0})}{\|u(t)\|}_{L^{p,q}(\mathbb{R}^{n})\cap L^{\overline{\infty}}(\mathbb{R}^{n})} < M.
    \end{equation}
    Since $\int_{0}^{T_{0}}{\|\nabla\mathcal{T}(s)\|}_{L^{1}(\mathbb{R}^{n})}\,\mathrm{d}s<\infty$, if $\epsilon>0$ is given then there exists $\eta>0$ such that
    \begin{equation}\label{uniform-integrability}
        \int_{s_{0}}^{s_{0}+\eta}{\|\nabla\mathcal{T}(s)\|}_{L^{1}(\mathbb{R}^{n})}\,\mathrm{d}s < \epsilon \quad \text{for all }0\leq s_{0}<s_{0}+\eta\leq T_{0}.
    \end{equation}
    We know that the function $\nabla\mathcal{T}:(0,\infty)\rightarrow L^{1}(\mathbb{R}^{n})$ is continuous, so the function $\nabla\mathcal{T}:[\frac{\eta}{2},T_{0}]\rightarrow L^{1}(\mathbb{R}^{n})$ is uniformly continuous. We can therefore choose $\delta\in(0,\frac{\eta}{2})$ such that
    \begin{equation}\label{uniform-continuity}
        {\|\nabla\mathcal{T}(t+h-s)-\nabla\mathcal{T}(t-s)\|}_{L^{1}(\mathbb{R}^{n})} \leq \epsilon \quad \text{for all } \frac{\eta}{2}<t-s<t+h-s<T_{0} \text{ with } h<\delta.
    \end{equation}
    Let $t_{0}<t<t+h<T_{0}$ with $h<\delta$. Then
    \begin{equation}
    \begin{aligned}
        &\quad z_{i}(t+h) - z_{i}(t) \\
        &= \int_{t_{0}}^{t+h}\nabla_{k}\mathcal{T}_{ij}(t+h-s)*[u_{j}(s)u_{k}(s)]\,\mathrm{d}s - \int_{t_{0}}^{t}\nabla_{k}\mathcal{T}_{ij}(t-s)*[u_{j}(s)u_{k}(s)]\,\mathrm{d}s \\
        &= \int_{t}^{t+h}\nabla_{k}\mathcal{T}_{ij}(t+h-s)*[u_{j}(s)u_{k}(s)]\,\mathrm{d}s + \int_{t_{0}}^{t}\left(\nabla_{k}\mathcal{T}_{ij}(t+h-s)-\nabla_{k}\mathcal{T}_{ij}(t-s)\right)*[u_{j}(s)u_{k}(s)]\,\mathrm{d}s \\
        &= I_{1} + I_{2}.
    \end{aligned}
    \end{equation}
    (If $t-t_{0}>\eta$, then we write $I_{2}=\int_{t_{0}}^{t}=\int_{t_{0}}^{t-\eta}+\int_{t-\eta}^{t}=I_{3}+I_{4}$). We use (\ref{uniform-integrability}) to estimate
    \begin{equation}
    \begin{aligned}
        {\|I_{1}\|}_{L^{p,q}(\mathbb{R}^{n})\cap L^{\overline{\infty}}(\mathbb{R}^{n})} &\leq \int_{t}^{t+h}{\|\nabla\mathcal{T}(t+h-s)\|}_{L^{1}(\mathbb{R}^{n})}{\|u(s)\otimes u(s)\|}_{L^{p,q}(\mathbb{R}^{n})\cap L^{\overline{\infty}}(\mathbb{R}^{n})}\,\mathrm{d}s \\
        &\leq M^{2}\int_{t}^{t+h}{\|\nabla\mathcal{T}(t+h-s)\|}_{L^{1}(\mathbb{R}^{n})}\,\mathrm{d}s \\
        &\leq M^{2}\epsilon.
    \end{aligned}
    \end{equation}
    If $t-t_{0}\leq\eta$, then (\ref{uniform-integrability}) yields
    \begin{equation}
    \begin{aligned}
        {\|I_{2}\|}_{L^{p,q}(\mathbb{R}^{n})\cap L^{\overline{\infty}}(\mathbb{R}^{n})} &\leq \int_{t_{0}}^{t}{\|\nabla\mathcal{T}(t+h-s)-\nabla\mathcal{T}(t-s)\|}_{L^{1}(\mathbb{R}^{n})}{\|u(s)\otimes u(s)\|}_{L^{p,q}(\mathbb{R}^{n})\cap L^{\overline{\infty}}(\mathbb{R}^{n})}\,\mathrm{d}s \\
        &\leq M^{2}\int_{t_{0}}^{t}\left({\|\nabla\mathcal{T}(t+h-s)\|}_{L^{1}(\mathbb{R}^{n})}+{\|\nabla\mathcal{T}(t-s)\|}_{L^{1}(\mathbb{R}^{n})}\right)\,\mathrm{d}s \\
        &\leq 2M^{2}\epsilon.
    \end{aligned}
    \end{equation}
    If $t-t_{0}>\eta$, then (\ref{uniform-continuity}) yields
    \begin{equation}
    \begin{aligned}
        {\|I_{3}\|}_{L^{p,q}(\mathbb{R}^{n})\cap L^{\overline{\infty}}(\mathbb{R}^{n})} &\leq \int_{t_{0}}^{t-\eta}{\|\nabla\mathcal{T}(t+h-s)-\nabla\mathcal{T}(t-s)\|}_{L^{1}(\mathbb{R}^{n})}{\|u(s)\otimes u(s)\|}_{L^{p,q}(\mathbb{R}^{n})\cap L^{\overline{\infty}}(\mathbb{R}^{n})}\,\mathrm{d}s \\
        &\leq M^{2}\int_{t_{0}}^{t-\eta}{\|\nabla\mathcal{T}(t+h-s)-\nabla\mathcal{T}(t-s)\|}_{L^{1}(\mathbb{R}^{n})}\,\mathrm{d}s \\
        &\leq (T_{0}-t_{0})M^{2}\epsilon,
    \end{aligned}
    \end{equation}
    and (\ref{uniform-integrability}) yields
    \begin{equation}
    \begin{aligned}
        {\|I_{4}\|}_{L^{p,q}(\mathbb{R}^{n})\cap L^{\overline{\infty}}(\mathbb{R}^{n})} &\leq \int_{t-\eta}^{t}{\|\nabla\mathcal{T}(t+h-s)-\nabla\mathcal{T}(t-s)\|}_{L^{1}(\mathbb{R}^{n})}{\|u(s)\otimes u(s)\|}_{L^{p,q}(\mathbb{R}^{n})\cap L^{\overline{\infty}}(\mathbb{R}^{n})}\,\mathrm{d}s \\
        &\leq M^{2}\int_{t-\eta}^{t}\left({\|\nabla\mathcal{T}(t+h-s)\|}_{L^{1}(\mathbb{R}^{n})}+{\|\nabla\mathcal{T}(t-s)\|}_{L^{1}(\mathbb{R}^{n})}\right)\,\mathrm{d}s \\
        &\leq 2M^{2}\epsilon.
    \end{aligned}
    \end{equation}
    For $t_{0}<t<t+h<T_{0}$ with $h<\delta$, we therefore have
    \begin{equation}
        {\|z(t+h)-z(t)\|}_{L^{p,q}(\mathbb{R}^{n})\cap L^{\overline{\infty}}(\mathbb{R}^{n})} \leq (3+T_{0}-t_{0})M^{2}\epsilon.
    \end{equation}
    This bound is independent of $t$, so we conclude that the function $z:(t_{0},T_{0})\rightarrow L^{p,q}(\mathbb{R}^{n})\cap L^{\overline{\infty}}(\mathbb{R}^{n})$ is continuous.
\end{proof}
\begin{theorem}(Continuity in space).
    Assume that $u$ satisfies \eqref{continuity-semigroup} for all $0<t_{0}<t<T$ and $x\in\mathbb{R}^{n}$, with ${\|u(t)\|}_{L^{\overline{\infty}}(\mathbb{R}^{n})}$ being locally bounded for $t\in(0,T)$. Then for all $\alpha\in(0,1)$ we have the H\"{o}lder estimate
    \begin{equation}
        {[u(t)]}_{C^{\alpha}(\mathbb{R}^{n})} \lesssim \frac{{\|u(t_{0})\|}_{L^{\overline{\infty}}(\mathbb{R}^{n})}}{{(t-t_{0})}^{\alpha/2}} + \int_{t_{0}}^{t}\frac{{\|u(s)\|}_{L^{\overline{\infty}}(\mathbb{R}^{n})}^{2}}{{(t-s)}^{(1+\alpha)/2}}\,\mathrm{d}s \quad \text{for all } 0<t_{0}<t<T.
    \end{equation}
\end{theorem}
\begin{proof}
    We write
    \begin{equation}
    \begin{aligned}
        u_{i}(t,x) - u_{i}(t,y) &= \int_{\mathbb{R}^{n}}\left[\Phi(t-t_{0},x-z)-\Phi(t-t_{0},y-z)\right]u_{i}(t_{0},z)\,\mathrm{d}z \\
        &\qquad - \int_{t_{0}}^{t}\int_{\mathbb{R}^{n}}\left[\nabla_{k}\mathcal{T}_{ij}(t-s,x-z)-\nabla_{k}\mathcal{T}_{ij}(t-s,y-z)\right][u_{j}(s,z)u_{k}(s,z)]\,\mathrm{d}z\,\mathrm{d}s \\
        &= I_{1} - I_{2} + I_{3} - I_{4} + I_{5} - I_{6},
    \end{aligned}
    \end{equation}
    where
    \begin{equation}
    \begin{aligned}
        \Omega &:= B(x,2R) \cup B(y,2R), \quad R := |x-y|, \\
        I_{1} &:= \int_{\Omega}\Phi(t-t_{0},x-z)u_{i}(t_{0},z)\,\mathrm{d}z, \\
        I_{2} &:= \int_{\Omega}\Phi(t-t_{0},y-z)u_{i}(t_{0},z)\,\mathrm{d}z, \\
        I_{3} &:= \int_{\mathbb{R}^{n}\setminus\Omega}\left[\Phi(t-t_{0},x-z)-\Phi(t-t_{0},y-z)\right]u_{i}(t_{0},z)\,\mathrm{d}z, \\
        I_{4} &:= \int_{t_{0}}^{t}\int_{\Omega}\nabla_{k}\mathcal{T}_{ij}(t-s,x-z)[u_{j}(s,z)u_{k}(s,z)]\,\mathrm{d}z\,\mathrm{d}s, \\
        I_{5} &:= \int_{t_{0}}^{t}\int_{\Omega}\nabla_{k}\mathcal{T}_{ij}(t-s,y-z)[u_{j}(s,z)u_{k}(s,z)]\,\mathrm{d}z\,\mathrm{d}s, \\
        I_{6} &:= \int_{t_{0}}^{t}\int_{\mathbb{R}^{n}\setminus\Omega}\left[\nabla_{k}\mathcal{T}_{ij}(t-s,x-z)-\nabla_{k}\mathcal{T}_{ij}(t-s,y-z)\right][u_{j}(s,z)u_{k}(s,z)]\,\mathrm{d}z\,\mathrm{d}s.
    \end{aligned}
    \end{equation}
    Using the pointwise estimate (\ref{heat-type-pointwise}) for $\nabla\mathcal{T}$, for all $\alpha\in(0,1)$ we have
    \begin{equation}
    \begin{aligned}
        |I_{4}| &\leq \int_{t_{0}}^{t}\left(\int_{\Omega}|\nabla\mathcal{T}(t-s,x-z)|\,\mathrm{d}z\right){\|u(s)\|}_{L^{\overline{\infty}}(\mathbb{R}^{n})}^{2}\,\mathrm{d}s \\
        &\lesssim \int_{t_{0}}^{t}\left(\int_{\Omega}{(t-s)}^{-\frac{n+1}{2}}{\left(1+\frac{|x-z|}{\sqrt{t-s}}\right)}^{-(n+1)}\,\mathrm{d}z\right){\|u(s)\|}_{L^{\overline{\infty}}(\mathbb{R}^{n})}^{2}\,\mathrm{d}s \\
        &\lesssim \int_{t_{0}}^{t}\left(\int_{0}^{3R}{(t-s)}^{-\frac{n+1}{2}}{\left(1+\frac{r}{\sqrt{t-s}}\right)}^{-(n+1)}r^{n-1}\,\mathrm{d}r\right){\|u(s)\|}_{L^{\overline{\infty}}(\mathbb{R}^{n})}^{2}\,\mathrm{d}s \\
        &= \int_{t_{0}}^{t}\left(\int_{0}^{3R}{\left(\sqrt{t-s}+r\right)}^{-(n+1)}r^{n-1}\,\mathrm{d}r\right){\|u(s)\|}_{L^{\overline{\infty}}(\mathbb{R}^{n})}^{2}\,\mathrm{d}s \\
        &\leq \int_{t_{0}}^{t}\left(\int_{0}^{3R}{\left(\sqrt{t-s}+r\right)}^{-2}\,\mathrm{d}r\right){\|u(s)\|}_{L^{\overline{\infty}}(\mathbb{R}^{n})}^{2}\,\mathrm{d}s \\
        &= \int_{t_{0}}^{t}\frac{3R}{\sqrt{t-s}(\sqrt{t-s}+3R)}{\|u(s)\|}_{L^{\overline{\infty}}(\mathbb{R}^{n})}^{2}\,\mathrm{d}s \\
        &\lesssim \int_{t_{0}}^{t}\frac{R^{\alpha}{\|u(s)\|}_{L^{\overline{\infty}}(\mathbb{R}^{n})}^{2}}{{(t-s)}^{(1+\alpha)/2}}\,\mathrm{d}s,
    \end{aligned}
    \end{equation}
    where in the last line we use the identity $\frac{b}{a+b}\leq{\left(\frac{b}{a+b}\right)}^{\alpha}\leq{\left(\frac{b}{a}\right)}^{\alpha}$ for $a,b\in(0,\infty)$ and $\alpha\in(0,1)$. Similarly we obtain
    \begin{equation}
    \begin{aligned}
        |I_{5}| &\lesssim  \int_{t_{0}}^{t}\frac{R^{\alpha}{\|u(s)\|}_{L^{\overline{\infty}}(\mathbb{R}^{n})}^{2}}{{(t-s)}^{(1+\alpha)/2}}\,\mathrm{d}s, \\
        |I_{1}|\vee|I_{2}| &\lesssim \log\left(1+\frac{3R}{\sqrt{t-t_{0}}}\right){\|u(t_{0})\|}_{L^{\overline{\infty}}(\mathbb{R}^{n})} \lesssim \frac{R^{\alpha}{\|u(t_{0})\|}_{\mathcal{B}(\mathbb{R}^{n})}}{{(t-t_{0})}^{\alpha/2}}.
    \end{aligned}
    \end{equation}
    Using the mean value theorem, the pointwise estimate (\ref{heat-type-pointwise}) for $\nabla^{2}\mathcal{T}$, and the implication
    \begin{equation}
        |x-\xi|\leq R\leq \frac{1}{2}|x-z| \quad \Rightarrow \quad |\xi-z|\geq\frac{1}{2}|x-z|,
    \end{equation}
    for all $\alpha\in(0,1)$ we have (for some $\xi(z)$ between $x$ and $y$)
    \begin{equation}
    \begin{aligned}
        |I_{6}| &\leq \int_{t_{0}}^{t}\left(\int_{\mathbb{R}^{n}\setminus\Omega}R|\nabla^{2}\mathcal{T}(t-s,\xi(z)-z)|\,\mathrm{d}z\right){\|u(s)\|}_{L^{\overline{\infty}}(\mathbb{R}^{n})}^{2}\,\mathrm{d}s \\
        &\lesssim \int_{t_{0}}^{t}\left(\int_{\mathbb{R}^{n}\setminus\Omega}R{(t-s)}^{-\frac{n+2}{2}}{\left(1+\frac{|\xi(z)-z|}{\sqrt{t-s}}\right)}^{-(n+2)}\,\mathrm{d}z\right){\|u(s)\|}_{L^{\overline{\infty}}(\mathbb{R}^{n})}^{2}\,\mathrm{d}s \\
        &\lesssim \int_{t_{0}}^{t}\left(\int_{2R}^{\infty}R{(t-s)}^{-\frac{n+2}{2}}{\left(1+\frac{r}{2\sqrt{t-s}}\right)}^{-(n+2)}r^{n-1}\,\mathrm{d}r\right){\|u(s)\|}_{L^{\overline{\infty}}(\mathbb{R}^{n})}^{2}\,\mathrm{d}s \\
        &\sim \int_{t_{0}}^{t}\left(\int_{2R}^{\infty}R{\left(2\sqrt{t-s}+r\right)}^{-(n+2)}r^{n-1}\,\mathrm{d}r\right){\|u(s)\|}_{L^{\overline{\infty}}(\mathbb{R}^{n})}^{2}\,\mathrm{d}s \\
        &\leq \int_{t_{0}}^{t}\left(\int_{2R}^{\infty}R{\left(2\sqrt{t-s}+r\right)}^{-3}\,\mathrm{d}r\right){\|u(s)\|}_{L^{\overline{\infty}}(\mathbb{R}^{n})}^{2}\,\mathrm{d}s \\
        &\sim \int_{t_{0}}^{t}\frac{R}{{(2\sqrt{t-s}+2R)}^{2}}{\|u(s)\|}_{L^{\overline{\infty}}(\mathbb{R}^{n})}^{2}\,\mathrm{d}s \\
        &\lesssim \int_{t_{0}}^{t}\frac{R^{\alpha}{\|u(s)\|}_{L^{\overline{\infty}}(\mathbb{R}^{n})}^{2}}{{(t-s)}^{(1+\alpha)/2}}\,\mathrm{d}s,
    \end{aligned}
    \end{equation}
    where in the last line we use the identity $\frac{b}{{(a+b)}^{2}}=\frac{1}{(a+b)}\frac{b}{a+b}\leq\frac{1}{a}{\left(\frac{b}{a}\right)}^{\alpha}$. Similarly we obtain
    \begin{equation}
        |I_{3}| \lesssim \frac{R}{(2\sqrt{t-s}+2R)}{\|u(t_{0})\|}_{L^{\overline{\infty}}(\mathbb{R}^{n})} \lesssim \frac{R^{\alpha}{\|u(t_{0})\|}_{L^{\overline{\infty}}(\mathbb{R}^{n})}}{{(t-t_{0})}^{\alpha/2}}.
    \end{equation}
\end{proof}

\end{document}